\documentclass[11pt]{article}

\usepackage[english]{babel}
\usepackage{hyperref}

\newcommand{\insieme}[1]{\left\{ #1 \right\}}
\usepackage{fullpage}

\usepackage{enumerate}
\usepackage{amsmath,amsfonts,amssymb,amsthm, bbold}
\usepackage{svg}
\usepackage{graphics,esint}
\usepackage{epsfig}
\usepackage{xcolor}
\usepackage{xspace}

\usepackage{xfrac}
\usepackage{comment}

\usepackage{cleveref}

\usepackage{graphicx}
\allowdisplaybreaks 

\crefname{section}{§}{§§}
\Crefname{section}{§}{§§}

\newcommand\scalare[1]{{\left\langle #1 \right\rangle}}

\def\lt{\left}
\def\vol{\mathrm{vol}}

\def\rt{\right}

\def\Acal{\mathcal{A}}

\def\rhs{r.h.s.\xspace}
\def\lhs{l.h.s.\xspace}
\def\st{s.t.\xspace}

\makeatletter
\newtheorem*{rep@theorem}{\rep@title}
\newcommand{\newreptheorem}[2]{%
\newenvironment{rep#1}[1]{%
 \def\rep@title{#2 \ref{##1}}%
 \begin{rep@theorem}}%
 {\end{rep@theorem}}}
\makeatother

\newtheorem{theorem}{Theorem}[section]

\newtheorem{lemma}[theorem]{Lemma}
\newtheorem{definition}[theorem]{Definition}

\newtheorem{remark}[theorem]{Remark}

\newtheorem{proposition}[theorem]{Proposition}

\newtheorem{corollary}[theorem]{Corollary}

\newreptheorem{theorem}{Theorem}

\newcommand{\R}{\mathbb{R}}
\newcommand{\Z}{\mathbb{Z}}

\newcommand{\Fcal}{\mathcal{F}}

\newcommand{\Ka}{K_{\alpha, \tau}}
\newcommand{\Kah}{\widehat{K}_{\alpha,\tau}}

\def\N{\mathbb N}
\def\eps{\varepsilon}

\def\per{\mathrm{Per}}

\def\eps{\varepsilon}

\def\pa{p(\alpha)}
\def\ba{\beta(\alpha)}

\def\d {\,\mathrm {d}}
\def\dx{\,\mathrm {d}x}
\def\dz{\,\mathrm {d}z}
\def\ds{\,\mathrm {d}s}
\def\du{\,\mathrm {d}u}
\def\dv{\,\mathrm {d}v}
\def\dt{\,\mathrm {d}t}
\def\dy{\,\mathrm {d}y}

\def\FtL{\mathcal F_{\tau,L}}

\numberwithin{equation}{section}

\usepackage{authblk}

\author{Alicja Kerschbaum	\footnote{kerschbaum@math.fau.de}}
\affil{Friedrich Alexander Universit\"at  Erlangen--N\"urnberg }

\title{Striped patterns for generalized antiferromagnetic functionals with power law kernels of exponent smaller than $d+2$}

\date{}

\begin{document}

\maketitle

\begin{abstract}
	We consider a class of continuous generalized antiferromagnetic models previously studied in \cite{GR} and \cite{DR},  and in the discrete in \cite{2011PhRvB..84f4205G,2014CMaPh.tmp..127G,GiuSeirGS}.
	The functional consists of an anisotropic perimeter term and a repulsive nonlocal term with a power law kernel. In certain regimes the two terms enter in competition and symmetry breaking with formation of periodic striped patterns is expected to occur. 
	
	In this paper we extend the results of \cite{DR} to power law kernels within  a  range of exponents smaller than $d+2$, being $d$ the dimension of the underlying space. In particular, we prove that in a suitable regime minimizers are periodic unions of stripes with a given optimal period. Notice that the exponent $d+1$ corresponds to an anisotropic version of the model for pattern formation in thin magnetic films.    
\end{abstract}

\section{Introduction}

In this paper we consider the following functional: for a set $E\subset\R^d$, $\alpha<1$, $d\geq1$, $L>0$ define
\begin{equation}\label{E:F}
\tilde{\mathcal F}_{\alpha,J,L}(E)=\frac{1}{L^d}\Big(J\per_{1}(E,[0,L)^d)-\int_{[0,L)^d}\int_{\R^d}{|\chi_E(x)-\chi_E(y)|}{K_{\alpha,1}(x-y)}\dy\dx\Big),
\end{equation}
where $J$ is a positive constant,
\begin{equation*} 
\begin{split}
\per_{1}(E,[0,L)^d):=\int_{\partial E\cap [0,L)^d}|\nu^E(x)|_1\, d\mathcal H^{d-1}(x),\quad \text{$|z|_1=\sum_{i=1}^d|z_i|$},
\end{split}
\end{equation*} 
with $\nu^E(x)$ exterior normal to $E$ in $x$, is the $1$-perimeter of $E$ and $K_{\alpha,1} (\zeta)$ is a power law kernel 
\begin{equation}\label{e:k_1} 
\begin{split}
K_{\alpha,1} (\zeta) = \frac{1}{(|\zeta|_{1} + 1)^{p(\alpha)}},\qquad
p(\alpha)=d+2-\alpha.
\end{split}
\end{equation} 

Such a functional belongs to a class of antiferromagnetic potentials first considered in the discrete setting  by Giuliani, Lebowitz, Lieb and Seiringer (cf.~\cite{2011PhRvB..84f4205G,2014CMaPh.tmp..127G,GiuSeirGS})
and then in the continuous setting by Daneri, Goldman, Runa (cf.~\cite{GR,DR,DR2}).

For suitable values of $J$ the local attractive term enters in competition with the nonlocal repulsive term and as a result the minimizers of \eqref{E:F} are expected to form periodic patterns.

In particular,  there exists a critical constant
\begin{equation}
J_c=\int_{ \R^d}|\zeta_1|K_{\alpha,1}(\zeta)\d\zeta
\label{eq:jc}
\end{equation}
such that the following holds. If  $J>J_c$ then the perimeter term prevails and then the global minimizers are trivial, namely either empty or the whole domain. If $J\in[J_c-\tau,J_c)$ for $0<\tau\ll1$ instead, minimizers are expected to be periodic unions of stripes with boundaries orthogonal to some Euclidean coordinate  and of the same width and distance. 
By union of stripes we mean a $[0,L)^d$-periodic set which is, up to Lebesgue null sets, of the form $V_i^\perp+\bar Ee_i$ for some $i\in\{1,\dots,d\}$, where $\{ e_i\}_{i=1}^d$ is the canonical basis, $V_i^\perp$ is the $(d-1)$-dimensional subspace orthogonal to $e_i$ and  $\bar{E}\subset \R$ with $\bar E\cap [0,L)=\cup_{k=1}^N(s_k,t_k)$. 
We say that a union of stripes is periodic if $\exists\, h>0$, $\nu\in\R$ s.t. $\bar E\cap [0,L)=\cup_{k=0}^N(2kh+\nu,(2k+1)h+\nu)$, namely if $t_k-s_k=s_{k}-t_{k-1}=h$ for all $k$.

If $d=1$, the conjecture has been proved to hold in many instances (see e.g. \cite{Mu,CO,GLL1d}).

As soon as $d\geq2$ showing pattern formation is a rather challenging problem due to the phenomenon of symmetry breaking: minimizers (in this case unions of stripes) have less symmetries than the functional itself (in this case symmetry w.r.t. coordinate exchange).
When the anisotropic norm in the perimeter and in the kernel is replaced by the Euclidean norm (thus the potential is rotational symmetric) and $\alpha=4$ (i.e., $p(\alpha)=d-2$) the interaction functional \eqref{E:F} corresponds to the sharp interface version  of the celebrated Ohta-Kawasaki model for diblock copolymers \cite{OhtaKawasaki}. This model is well-studied (see e.g. \cite{ACO,CiSp,choksi2010small,KnMu,KNM,KMN2, GMS,MoriniSternberg, pasta, PV}) but though pattern formation is observed in experiments and numerical simulations (\cite{MR2338353,Seul476}), a rigorous proof is still largely open (see \cite{MoriniSternberg} for a proof of striped patterns in thin $2D$ domains and \cite{PV} for an asymptotic limit). 

Another physically interesting regime at which pattern formation occurs is the one corresponding to $\alpha=1$ (i.e., $p(\alpha)=d+1$). The corresponding model is used to describe thin magnetic films (\cite{CL, HS,  H}). Also in this case a full characterization of the minimizers is still  missing (for partial results on the problem or related ones see e.g. \cite{KMNo,ChoKo,ChoKoO, MuSi, KDB}). An interesting related problem would be to look for phase separation in non-flat surfaces in $\R^3$ without interpenetration of matter as considered in \cite{OR}. 

The only full characterizations of the minimizers of \eqref{E:F} as stripes have been given for lower values of $\alpha$ (i.e., higher values of $p(\alpha)$). The first result of this type was given in the discrete setting by Giuliani and Seiringer in \cite{GiuSeirGS} for $\alpha<2-d$ (i.e., $p(\alpha)>2d$). In the continuous setting, Goldman and Runa \cite{GR} proved that stripes patterns occur at the $\Gamma$-limit as the constant $J$ converges to the critical constant.  In \cite{DR} Daneri and Runa proved that, whenever  $\alpha\leq0$, minimizers of \eqref{E:F} are stripes for $J\in[J_c-\tau,J_c)$ for some positive $\bar{\tau} > 0$ and all $\tau < \bar{\tau}$. Moreover, they proved the so-called thermodynamic limit, namely that the value of such $\tau$ can be chosen independently of $L$, as soon as $L$ is an even multiple of the expected width of the optimal stripes. The strategy adopted in \cite{DR} is rather different from the one in \cite{GiuSeirGS} since in the continuous setting defects can carry an arbitrarily small energy and moreover larger values of $\alpha$ ($\alpha\leq0$ instead of $\alpha<2-d$) require different estimates. In \cite{DKR} one-dimensionality and periodicity of minimizers of the corresponding diffuse-interface functional was proved. In \cite{DR3} striped pattern formation was proved when the nonlocal term is defined through the Yukawa (screened Coulomb) kernel, used to model colloidal systems and protein solutions (see also \cite{DR2, DR4})

Our aim in this paper is to extend the results of \cite{DR} to a range of exponents $\alpha>0$. In order to state our results precisely,  it is convenient to rescale the functional \eqref{E:F} in such a way that  the width and the energy of the stripes are of order $O(1)$ for $\tau=J_c-J$ small.  
Let $\alpha \in [0,1)$ and  $\ba = \pa - d - 1 = 1 - \alpha$. The scaling is defined by: 
$$x = \tau^{-1/\ba} \hat{x}, \qquad L = \tau^{-1/\ba}\hat{L}, \qquad \Fcal(E) = \tau\tau^{1/\ba} \hat{\Fcal}(\hat{E}). $$
Making the substitutions in \eqref{E:F}, letting also $\zeta=\tau^{-1\slash\beta}\tilde{\zeta}$ and in the end dropping the tildes, one has that the rescaled functional that we want to study is given by the formula \eqref{E:F2}.

\begin{equation}\label{E:F2}
\begin{split}
\mathcal F_{\alpha,\tau,L}(E)=\frac{1}{L^d}\Big(-\per_1(E,[0,L)^d)&+\int_{\R^d} K_{\alpha,\tau}(\zeta) \Big[\int_{\partial E \cap [0,L)^d} \sum_{i=1}^d|\nu^E_i(x)| |\zeta_i|\d\mathcal H^{d-1}(x)\\
&-\int_{[0,L)^d}|\chi_E(x)-\chi_E(x+\zeta)|\dx\Big]\d\zeta\Big),
\end{split}
\end{equation}
where  $$K_{\alpha,\tau}(\zeta)=\tau^{-p(\alpha)/\beta(\alpha)}K_{\alpha,1}(\zeta\tau^{-1/\beta(\alpha)}) = \frac{1}{(|\zeta|_{1} + \tau^{1/\beta(\alpha})^{p(\alpha)}}. $$
The first two terms come from rescaling of $J\per_1(E, [0,L)^d)$, where we set $J = J_C - \tau$ and we write explicitly $J_C$ as in \eqref{eq:jc}.
  For fixed $\alpha,\tau > 0$, consider first for all $L > 0$ the minimal value obtained by $\Fcal_{\alpha,\tau,L}$ on $[0,L)^d$-periodic stripes
and then the minimal among these values as $L$ varies in $(0,+\infty)$. We will denote this minimal energy value  by $C^*_{\alpha,\tau}$ and remark that this value is negative.

By the reflection positivity technique, this value is attained on periodic stripes of certain  characteristic  widths and distances. As in Theorem 1.1 of \cite{DR} for the case in which $\alpha\leq0$, one can prove that for all $\alpha_0<1$ there exist $\hat{\tau}>0$ such that for all $0<\tau\leq\hat{\tau}$ and for all $0\leq\alpha\leq\alpha_0$  there exists a unique optimal width for  the periodic stripes (namely, on which the value $C^*_{\alpha,\tau}$ is obtained), which we denote by $h^*_{\alpha,\tau}$. 

We can now state our main result.

\begin{theorem}
	\label{thm:main}
	     Let $d\geq1$, $p(\alpha)=d+2-\alpha$ and $h^{*}_{\alpha,\tau}$ be the optimal stripes' width for fixed $\alpha$ and $\tau$. 
	Then there exist $0<\bar\tau\leq\hat{\tau}$ and $0<\bar{\alpha}<1$ such that for every $0<\tau\leq \bar \tau$, and $0\leq\alpha\leq\bar{\alpha}$ one has that for every $k\in \N$ and  $L = 2k h_{\alpha,\tau}^{*}$,   the minimizers $E_{\alpha,\tau}$ of $\Fcal_{\alpha,\tau,L}$ are periodic optimal stripes of optimal width $h_{\alpha,\tau}^{*}$. 
\end{theorem}

In Theorem \ref{thm:main} we extend then the results of \cite{DR}, valid for all $\alpha\leq0$, to the range $(0,\bar{\alpha}]$, for some $\bar{\alpha}$ strictly positive and strictly smaller than one.

The overall strategy follows closely the one introduced for the proof of the analogous result in \cite{DR} (i.e., Theorem 1.4). However, the fundamental difference w.r.t. \cite{DR} is that the proof of the rigidity estimate for the $\Gamma$-limit of $\Fcal_{\alpha,\tau,L}$ as $\tau\to0$ (which is fundamental to show the symmetry breaking in \cite{DR}) is sharp for $\alpha=0$ (i.e., $p(\alpha)=d+2$) and cannot be extended to strictly positive values of $\alpha$.  Thus, we have to perform a double $\Gamma$-limit in $\tau$ and $\alpha$ in order to get, basing on the results in \cite{DR} (Proposition 3.2), the symmetry breaking for positive values of $\alpha$ which are sufficiently close to $0$. Our second main observation is that, once the symmetry breaking is proved (Proposition \ref{lemma:local_rigidity_alpha}), the stability argument showing that once close in $L^1$ to unions of stripes it is not convenient to deviate from being a union of stripes (Proposition \ref{lemma:stimaContributoVariazionePiccola}) works for any exponent $\alpha<1$ (recall that $\alpha=1$ corresponds to the model for thin magnetic films). Moreover, while  extending all the necessary arguments and lemmas we compute explicitly the constants involved and their dependence on $\alpha$. 

\section{Notation and preliminary results}

Let $d \geq 1$. On $\R^d$ let us denote by $\langle\ldots\rangle$ the Euclidean scalar product and by $|.|$ the Euclidean norm. Let $e_1, \dots, e_n$ be the canonical basis on $\R^d$. We will often employ slicing arguments, for this reason we need definitions concerning the $i$-th component. For $x \in \R^d$ let $x_i = e_i\langle x,e_i\rangle $ and $x_i^{\perp} := x - x_i$.  
Let 
$|x|_1=\sum_{i=1}^d|x_i|$ be the $1$-norm  and $|x|_\infty=\max_i|x_i|$ the $\infty$-norm. 
While writing slicing formulas,
with a slight abuse of notation we will sometimes identify 
$x_i \in [0,L)^d$ with its coordinate in $\R$ w.r.t. 
$e_i$ and $\{x_i^{\perp}:\,x\in[0,L)^d\}$ with 
$[0,L)^{d-1}\subset\R^{d-1}$ 
so that $x_i \in [0,L)$ and $x_i^{\perp} \in [0,L)^{d-1}$. 

Whenever $\Omega\subset\R^d$ is a measurable set, we denote by $\mathcal H^{d-1}(\Omega)$ its $(d-1)$-dimensional Hausdorff measure and by $|\Omega|$ its Lebesgue measure.

Given a measure $\mu$ on $\R^d$, we denote by $|\mu|$ its total variation. 

We recall that  a set $E\subset\R^d$ is of (locally) finite perimeter if the distributional derivative of its characteristic function $\chi_E$ is a (locally) finite measure. We denote by $\partial E$ the reduced boundary of $E$ and by $\nu^E$ its exterior normal. 

The anisotropic $1$-perimeter of $E$ is given by 
\[
\per_1(E,[0,L)^d):=\int_{\partial E\cap [0,L)^d}|\nu^E(x)|_1\d\mathcal H^{d-1}(x)
\]
and, for $i\in\{1,\dots,d\}$ 
\begin{equation}
\label{eq:perI}
\per_{1i}(E,[0,L)^d)=\int_{\partial E\cap [0,L)^d}|\nu^E_i(x)|\d\mathcal H^{d-1}(x),
\end{equation}
thus $\per_1(E,[0,L)^d)=\sum_{i=1}^d\per_{1i}(E,[0,L)^d)$. 

For $i\in\{1,\dots,d\}$,  we define the one-dimensional slices of $E\subset\R^d$ in direction $e_i$ by

\[
E_{x_i^\perp}:=\bigl\{s\in[0,L):\,se_i+ x_i^\perp\in E\bigr\}, \qquad x_i^{\perp} \in [0,L)^{d-1}.
\]

Whenever $E$ is a set of locally finite perimeter, for a.e. $x_i^\perp$ its slice $E_{x_i^\perp}$ is a set of locally finite perimeter in $\R$ and the following slicing formula  holds for every $i\in\{1,\dots,d\}$
\[
\per_{1i}(E,[0,L)^d)=\int_{\partial E\cap [0,L)^d}|\nu^E_i(x)|\d\mathcal H^{d-1}(x)=\int_{[0,L)^{d-1}}\per_1(E_{x_i^\perp},[0,L))\dx_i^\perp.
\]

When $d=1$ one can define 
\[
\per_1(E, [0,L))  =\per(E,[0,L))= \#(\partial E \cap [0,L) ),
\]
where $\partial E$ is the reduced boundary of $E$.

Consider  $E\subset\R$  a set of locally finite perimeter and $s\in\partial E$ a point in the relative boundary of $E$. We will denote by 
\begin{equation}
\label{eq:s+s-}
\begin{split}
s^+ &:= \inf\{ t' \in \partial E, \text{with } t' > s  \} \\ s^- &:= \sup\{ t' \in \partial E, \text{with } t' < s  \}. 
\end{split}
\end{equation}

We will also apply slicing on small cubes, depending on $l$, around a point. Therefore we introduce the following notation.
For $r> 0$ and $x^{\perp}_i$ we let $Q_{r}^{\perp}(x^\perp_{i}) = \{z^\perp_{i}:\, |x^{\perp}_{i} - z^{\perp}_{i} |_\infty \leq r  \}$ or we think of $x_i^\perp\in[0,L)^{d-1}$ as element of a lower dimensional space and $Q_r^\perp(x_i^\perp)$ as a subset of $\R^{d-1}$. 
We denote also by $Q^i_r(t_i)\subset\R$ the interval of length $r$ centred in $t_i$.

From \cite{GR,DR} we recall that the following lower bound holds for all $\alpha < 1$.

\begin{align}
\label{eq:gstr1}
\Fcal_{\alpha,\tau,L} (E) 
&\geq-\frac{1}{L^d}\sum_{i=1}^{d}\per_{1i}(E,[0,L)^d) + \frac{1}{L^d}\sum_{i=1}^{d} \Big[\int_{[0,L)^d\cap \partial E} \int_{\R^d} |\nu^{E}_{i} (x)| |\zeta_{i} | K_{\alpha,\tau}(\zeta) \d\zeta \d\mathcal H^{d-1}(x) \notag\\ 
&- \int_{[0,L)^d } \int_{\R^d} |\chi_{E}(x + \zeta_i)  - \chi_{E}(x) | K_{\alpha,\tau}(\zeta) \d\zeta\dx \Big] \notag\\ 
&+ \frac{2}{d} \frac{1}{L^d}\sum_{i=1}^d \int_{[0,L)^d } \int_{\R^d} |\chi_{E}(x + \zeta_{i}) -\chi_{E}(x) | | \chi_{E}(x + \zeta^{\perp}_i) - \chi_{E}(x) | K_{\alpha,\tau}(\zeta)
\d\zeta \dx.
\end{align}

Notice that in \eqref{eq:gstr1} equality holds whenever the set $E$ is a union of stripes.  Thus, proving that optimal unions of stripes are the minimizers of the r.h.s. of \eqref{eq:gstr1} implies that they are the minimizers for $\Fcal_{\alpha,\tau,L}$. 

Let us define
\[
\widehat K_{\alpha,\tau}(\zeta_i)=\int_{ \R^{d-1}}K_{\alpha,\tau}(\zeta_i+\zeta_i^\perp)\d\zeta_i^\perp.
\]

As in Section 7 of \cite{DR} we further decompose the r.h.s. of \eqref{eq:gstr1} as follows.

\begin{equation}\label{eq:decomp_r}
\begin{split}
&-\frac{1}{L^d}\per_{1i}(E,[0,L)^d)+ \frac{1}{L^d}\Big[\int_{[0,L)^d\cap \partial E} \int_{\R^d} |\nu^{E}_{i} (x)| |\zeta_{i} | K_{\alpha,\tau}(\zeta)\d\zeta\d\mathcal H^{d-1}(x) \\& - \int_{[0,L)^d } \int_{\R^d} |\chi_{E}(x + \zeta_i)  - \chi_{E}(x) | K_{\alpha,\tau}(\zeta) \d\zeta \dx \Big]
= \frac{1}{L^d}\int_{[0,L)^{d-1}}  \sum_{s\in  \partial E_{t^{\perp}_{i}}\cap [0,L]} r_{i,\alpha,\tau}(E,t_{i}^{\perp},s)
\dt^{\perp}_i, 
\end{split}
\end{equation}
where  for $s\in \partial E_{t_{i}^{\perp}}$ 
\begin{equation}\label{eq:ritau}
   \begin{split}
      r_{i,\alpha,\tau}(E, t_{i}^{\perp},s) := -1 + \int_{\R} |\zeta_{i}| \Kah (\zeta_{i})\d\zeta_i &- \int_{s^-}^{s}\int_{0}^{+\infty} |\chi_{E_{t_{i}^{\perp}}}(u + \rho) - \chi_{E_{t_{i}^{\perp}}}(u) | \Kah (\rho)\d\rho \du  \\ & - \int_{s}^{s^+}\int_{-\infty}^{0} |\chi_{E_{t_{i}^{\perp}}}(u + \rho) - \chi_{E_{t_{i}^{\perp}}}(u) |\Kah (\rho) \d\rho \du.\\ 
   \end{split}
\end{equation}
and $s^-<s<s^+$ are as in \eqref{eq:s+s-}.

Defining
\begin{equation} 
\label{eq:defFE}
\begin{split}
f_{E}(t_{i}^{\perp},t_{i},\zeta_{i}^{\perp},\zeta_{i}): =  |\chi_{E}(t_i+t_i^\perp + \zeta_{i} )  - \chi_{E}(t_i+t_i^\perp)| |\chi_{E}(t_i+t_i^\perp + \zeta^{\perp}_{i}) - \chi_{E}(t_i+t_i^\perp) |,
\end{split}
\end{equation} 
one has that
\begin{align}
\label{eq:decomp_double_prod}
 \frac{2}{d} \frac{1}{L^d}\sum_{i=1}^d &\int_{[0,L)^d } \int_{\R^d} |\chi_{E}(x + \zeta_{i}) -\chi_{E}(x) | | \chi_{E}(x + \zeta^{\perp}_i) - \chi_{E}(x) | K_{\alpha,\tau}(\zeta)
\d\zeta \dx=\notag\\
&=
\sum_{i=1}^d
\frac{2}{d}\frac{1}{L^d} \int_{[0,L)^d } \int_{\R^d} f_{E}(t_{i}^{\perp},t_{i},\zeta_{i}^{\perp},\zeta_{i}) K_{\alpha,\tau}(\zeta)  \d\zeta \dt\notag\\
& =
\sum_{i=1}^d
\bigl[ \frac{1}{L^d}  \int_{[0,L)^{d-1}} \sum_{s\in \partial E_{t_{i}^{\perp}}\cap [0,L]} v_{i,\alpha,\tau}(E,t_{i}^{\perp},s)\dt_{i}^\perp + \frac{1}{L^d}\int_{[0,L)^d} w_{i,\alpha,\tau}(E,t_i^\perp,t_i) \dt \bigr],
\end{align}
where
\begin{equation}\label{eq:witau}
      {w}_{i,\alpha, \tau}(E,t_{i}^{\perp},t_{i}) = \frac{1}{d}\int_{\R^d}  
      f_{E}(t_{i}^{\perp},t_{i},\zeta_{i}^{\perp},\zeta_{i}) \Ka(\zeta)  \d\zeta. 
\end{equation}
and
\begin{equation}\label{eq:vitau}
   v_{i,\alpha,\tau}(E,t_{i}^{\perp},s) =  \frac{1}{2d}\int_{s^{-}}^{s^{+}} \int_{\R^{d}} f_{E}(t^{\perp}_{i},u,\zeta^{\perp}_{i},\zeta_{i}) \Ka(\zeta) \d\zeta\du.
\end{equation}

Hence, putting together \eqref{eq:decomp_r} and \eqref{eq:decomp_double_prod} one has the following decomposition

\begin{align}
\label{eq:decomposition}
\Fcal_{\alpha,\tau,L} (E) 
&\geq 
\sum_{i=1}^d \frac{1}{L^d}\int_{[0,L)^{d-1}}  \sum_{s\in  \partial E_{t^{\perp}_{i}}\cap [0,L]} r_{i,\alpha,\tau}(E,t_{i}^{\perp},s)
\dt^{\perp}_i\notag\\
&+\sum_{i=1}^d \frac{1}{L^d}  \int_{[0,L)^{d-1}} \sum_{s\in \partial E_{t_{i}^{\perp}}\cap [0,L]} v_{i,\alpha,\tau}(E,t_{i}^{\perp},s)\dt_{i}^\perp \notag\\
&+ \sum_{i=1}^d \frac{1}{L^d}\int_{[0,L)^d} w_{i,\alpha,\tau}(E,t_i^\perp,t_i) \dt.
\end{align}

The term $r_{i,\alpha,\tau}$ penalizes oscillations with high frequency  in direction $e_i$, namely sets $E$ whose slices in direction $e_i$ have boundary points at small minimal distance (see Lemma \ref{rmk:stimax1}). The term $v_{i,\alpha,\tau}$  penalizes oscillations in direction $e_i$ whenever the neighbourhood of the point in $\partial E\cap{Q_l(z)}$ is close in $L^1$ to a stripe oriented along $e_j$ (see Proposition \ref{lemma:stimaContributoVariazionePiccola}). 

For every cube  $Q_{l}(z)$, with $l<L$ and $z\in[0,L)^d$, define now the following localization of $\Fcal_{\alpha,\tau,L}$
\begin{equation} 
\label{eq:fbartau}
\begin{split}
\bar{F}_{i,\alpha,\tau}(E,Q_{l}(z)) &:= \frac{1}{l^d  }\Big[\int_{Q^{\perp}_{l}(z_{i}^{\perp})} \sum_{\substack{s \in \partial E_{t_{i}^{\perp}}\\ t_{i}^{\perp}+se_i\in Q_{l}(z)}} (v_{i,\alpha,\tau}(E,t_{i}^{\perp},s)+ r_{i,\alpha,\tau}(E,t_{i}^{\perp},s)) \dt_{i}^{\perp} + \int_{Q_{l}(z)} {w_{i,\alpha,\tau}(E,t_{i}^{\perp}, t_i) \dt}\Big],\\
\bar{F}_{\alpha,\tau}(E,Q_{l}(z)) &:= \sum_{i=1}^d\bar F_{i,\alpha,\tau}(E,Q_{l}(z)).
\end{split}
\end{equation} 

The following inequality holds:
\begin{equation}
\label{eq:gstr14}
\begin{split}
\Fcal_{\alpha,\tau,L}(E) \geq \frac{1}{L^d} \int_{[0,L)^d}  
\bar{F}_{\alpha,\tau}(E,Q_{l}(z)) \dz. 
\end{split}
\end{equation}
Since in \eqref{eq:gstr14} equality holds for unions of stripes, in order to prove Theorem \ref{thm:main} one can reduce to show that the minimizers of its right hand side are periodic optimal stripes provided $\alpha$, $\tau$ and $L$ satisfy the conditions of the theorem. 

In the next definition we recall a quantity introduced in \cite{DR}[Definition 7.3] which measures the $L^1$ distance of a set from being a union of stripes.

\begin{definition}
   \label{def:defDEta}
   For every $\eta$ we denote by $\Acal^{i}_{\eta}$ the family of all sets $F$ such that  
   \begin{enumerate}[(i)]
      \item they are union of stripes oriented along the direction $e_i$ 
      \item their connected components of the boundary are distant at least $\eta$. 
   \end{enumerate}
   We denote by 
   \begin{equation} 
      \label{eq:defDEta}
      \begin{split}
         D^{i}_{\eta}(E,Q) := \inf\Big\{ \frac{1}{\vol(Q)} \int_{Q} |\chi_{E} -\chi_{F}|:\ F\in \Acal^{i}_{\eta} \Big\} \quad\text{and}\quad D_{\eta}(E,Q) = \inf_{i} D^{i}_{\eta}(E,Q).
      \end{split}
   \end{equation} 
   Finally, we let $\mathcal A_\eta:=\cup_{i}\mathcal A^i_{\eta}$.
\end{definition}

We recall also the following properties of the functional defined in \eqref{eq:defDEta}. 

\begin{remark} As observed in \cite{DR} Remark 7.4, the distance function from the set of stripes satisfies the following properties.
	\label{rmk:lip} \ 
	\begin{enumerate}[(i)]
		\item Let $E\subset\R^d$.  Then the map $z\mapsto D_{\eta}(E,Q_{l}(z))$ is Lipschitz, with Lipschitz constant $C_d/l$, where $C_d$ is a constant depending only on the dimension $d$. 
		
		In particular, whenever $D_{\eta}(E,Q_{l}(z)) > \alpha$ and $D_{\eta}(E,Q_{l}(z')) < \beta$,  then $|z - z'|> l(\alpha - \beta)/C_{d}$.

		\item
		For every $\varepsilon$ there exists ${\delta}_0= \delta_0(\varepsilon)$ such that  for every $\delta \leq \delta_0 $ whenever $D^{j}_{\eta}(E,Q_{l}(z))\leq \delta$ and $D^{i}_{\eta}(E,Q_{l}(z))\leq \delta$ with $i\neq j$ for some $\eta>0$,  it holds 
		\begin{equation}
		\label{eq:gsmstr2}
		\begin{split}
		\min\big(|Q_l(z)\setminus E|, |E \cap Q_l(z)| \big) \leq\varepsilon. 
		\end{split}
		\end{equation}
	\end{enumerate}
\end{remark}

\section{Monotonicity and scaling properties with respect to $\alpha$ and $\tau$}\label{sec:alpha}

In this section we study more in details some basic properties of the kernel in dependence of the parameters $\alpha$ and $\tau$. In particular, we explicitly compute fundamental constants which will be used during the proof of Theorem \ref{thm:main}, at the aim of controlling the dependency on $\alpha$. 

One has the following simple remark.
\begin{remark}
For $\alpha < 1$ and $0 < \tau < 1$ the kernel $K_{\alpha,\tau}$ is continuous in $\R^{d}\setminus \{0\}$. Moreover we have that
\begin{align}
\underset{\alpha \rightarrow 1}{\lim} K_{\alpha,\tau}(\zeta)&= \frac{1}{|\zeta|^{d+1}} \\
\underset{\tau \rightarrow 0}{\lim} K_{\alpha,\tau}(\zeta)&=\frac{1}{|\zeta|^{\pa}} \qquad\forall\,\alpha<1.
\end{align}
\end{remark}

In the following two lemmas we summarize some useful explicit computations and properties of the kernel.

\begin{lemma}
\label{lemma:propertiesKhat}
One has the following:
\begin{enumerate}
\item For every $\alpha<1$,  one has that for every $\tau>0$
\begin{equation}\label{eq:hatkformula}
\Kah(z)=\int_{ \R^{d-1}}\Ka(ze_i+z_i^\perp)\dz_i^\perp = 
C_{1, \alpha} \frac{1}{(|z| + \tau^{1/ \ba})^{q(\alpha)}},
\end{equation}
where 
\begin{equation}\label{eq:c1}
C_{1,\alpha}=\int_{\R^{d-1}} \frac{1}{(|\xi|_1 + 1)^{\pa}} \d\zeta
\end{equation} and $q(\alpha) = \pa - d + 1=3-\alpha>2$.\\
\item For every $\alpha<1$ and $\tau>0$, the function $\Kah$ is the Laplace transform of a positive function. \\
\item Let  $0<\tau' < \tau \leq 1$, $z \in \R^d$.  Then ${K}_{\alpha, \tau'}(z) > K_{\alpha,\tau}(z)$. \\
\item Let $\tau = 0$, $0\leq{\alpha}'< \alpha<1$ and $|z|<1$. Then,  ${K}_{\alpha', 0}(z) > {K}_{\alpha, 0}(z)$. 
\end{enumerate}
\end{lemma}

\begin{proof}
	Let us start to prove $1.$ One has that, by setting $\zeta=(|z|+\tau^{1/\ba})\xi$ and applying the usual change of variables formula 
\begin{align*}
\widehat{K}_{\alpha,\tau}(z)&=\int_{\R^{d-1}} \frac{1}{(|\zeta|_1 + |z| + \tau^{1 / \ba})^{\pa}} \d\zeta \\
 & = \int_{\R^{d-1}} \frac{(|z|+\tau^{1/\ba})^{d-1}}{(|z|+\tau^{1/\ba})^{\pa}(|\xi|_1 + 1)^{\pa}}  \d\xi  \\
 & =  C_{1,\alpha} \frac{1}{(|z| + \tau^{1/\ba})^{q(\alpha)}}
\end{align*}
where $C_{1,\alpha} = \int_{\R^{d-1}} \frac{1}{(|\xi|_1 + 1)^{\pa}} \d\xi$.

For the proof of $2.$ we refer to \cite{DR}, Section 4.

The proof of $3.$ and $4.$ follows easily from the definition of the kernel. 
\end{proof}

\begin{lemma}\label{lemma:ineq2}
Let $0 \leq\alpha < 1$ and $0 < \rho < 1$. One has the following:
\begin{enumerate}
\item 
\begin{align}\label{eq:onedimestkernel}
\int_{c}^{+\infty} \frac{z - c}{(z + \tau^{1/\ba})^{q(\alpha)}} \geq C_{2,\alpha} \min \{ \tau^{-1}, c^{-\ba} \}
\end{align}  
where  
\begin{equation}\label{eq:c2}
C_{2,\alpha} = \frac{1}{(q(\alpha)-1)(q(\alpha)-2)}.
\end{equation}

\item For all $\alpha_0<1$ there exist $\eps_1>0$ and $\tau_1>0$ such that for all $0<\tau\leq\tau_1$, $0\leq\alpha\leq\alpha_0$ and $0<\eps\leq\eps_1$ one has that
\begin{align}
\frac{7 C\varepsilon^{d+1}}{16 (\varepsilon^{\pa} + \tau^{\pa / \ba})} > 1, \qquad \text{where $C$ is such that } \qquad  
K_{\alpha,\tau}(\zeta)\geq \frac{C}{|\zeta |^{p(\alpha)} + \tau^{p(\alpha)/\beta(\alpha)}},
\end{align}
\item The constants $C_{1,\alpha}$ and $C_{2,\alpha}$ defined respectively in \eqref{eq:c1} and \eqref{eq:c2} are increasing in $\alpha$.
\end{enumerate}

\end{lemma}

\begin{proof}
Let us first prove $1.$ We have that
\begin{align*}
\int_{c}^{+\infty} \frac{z-c}{(z + \tau^{1/\ba})^{q(\alpha)}}\dz 
&= \int_{c + \tau^{1/\ba} }^{+\infty} \frac{z -
\tau^{1/\ba}-c}{z^{q(\alpha)}}\dz  
\\
&=
\frac{1}{2-q(\alpha)} [z^{2-{q(\alpha)}}]_{c+\tau^{1/\ba}}^{+\infty}
-
\frac{c+\tau^{1/\ba}}{1-q(\alpha)} [z^{1-q(\alpha)}]_{c+\tau^{1/\ba}}^{+\infty} \\
&=\frac{1}{(q(\alpha)-1)(q(\alpha)-2)}(c+\tau^{1/\ba})^{2-q(\alpha)}\\
&\geq \frac{1}{(q(\alpha)-1)(q(\alpha)-2)}\min \{ \tau^{-1}, c^{-\ba} \},
\end{align*}
where we used the fact that $q(\alpha)<2$ if $\alpha<1$ and the identity $2-q(\alpha)=-\ba$.

In order to prove 2., we notice that such $\eps_1,\tau_1$ exist since $\beta(\alpha)\geq\beta(\alpha_0)>0$ and 
\begin{equation}
0\leq\alpha\leq\alpha_0<1\quad\Rightarrow\quad p(\alpha)\geq p(\alpha_0)>d+1.
\end{equation} 
Notice that the constant $C$ in 2. is given by the equivalence of the 1-norm and the euclidean norm and the inequality $(a + b)^p \leq 2^{p-1}(a^p + b^p)$.

Point 3. follows directly from the definitions of $C_{1,\alpha}$ and $C_{2,\alpha}$. 
\end{proof}

\section{The one-dimensional problem}\label{sec:1D}

As in \cite{DR,DR3}, in the proof of Theorem \ref{thm:main} we will need to perform purely one-dimensional optimizations (see e.g. Lemma \ref{lemma:1D-optimization}). In particular, we need to know more precisely the dependence of the minimal energy value $C^*_{\alpha,\tau}$ of the one-dimensional problem on $\alpha$ and $\tau$ (see Corollary \ref{cor:cupper}).

For any $E\subset\R$ $L$-periodic set of finite perimeter we consider the one-dimensional functional

\begin{align}
\Fcal^{1d}_{\alpha,\tau,L}(E)&=\frac1L\Bigl(-\mathrm{Per}(E,[0,L))+\int_{ \R}\widehat{K}_{\alpha,\tau}(z)\Bigl[\mathrm{Per}(E,[0,L))|z|-\int_0^L|\chi_E(x)-\chi_E(x+z)|\dx\Bigr]\dz\Bigr)\notag\\
&=\frac1L\Bigl(-\mathrm{Per}(E,[0,L))+C_{1,\alpha}\int_{ \R}\frac{1}{(|z|+\tau^{1/\beta(\alpha)})^{q(\alpha)}}\Bigl[\mathrm{Per}(E,[0,L))|z|-\int_0^L|\chi_E(x)-\chi_E(x+z)|\dx\Bigr]\dz\Bigr),\label{eq:f1d}
\end{align}
where we used formula \eqref{eq:hatkformula}.

By the  reflection positivity technique, introduced in the context of quantum field theory \cite{osterNatale} and then applied for the first time in statistical mechanics in \cite{FroSim}, it is now a well known fact that minimizers of \eqref{eq:f1d} are periodic unions of stripes. For the first appearance of the technique in models with short-range and Coulomb type interactions see \cite{fro} and for further applications in one-dimensional models  \cite{GLL1d}, \cite{glllRP3} and  \cite{glllRP3bis}. 

The above motivates the following definition. For   $h>0$, let $E_h:=\cup_{k\in \Z} [(2k)h,(2k+1)h]$. Then, we define
\[e_{\alpha,\tau}(h):=\Fcal^{1d}_{\alpha,\tau,2h}(E_h)=\lim_{L\to+\infty} \Fcal^{1d}_{\alpha,\tau,L}(E_h).\]
One has the following
(for analogues computation when $\tau = 0$ see \cite{GR}):
\begin{lemma}\label{lemmahstar}
	For every $h>0$, it holds
	\begin{equation}\label{eq:einftyalpha}
	e_{\alpha,\tau}(h)=-\frac{1}{h}+\frac{C_{1,\alpha}C_{2,\alpha}}{h}\Bigl[2(h+\tau^{1/\beta(\alpha)})^{-(q(\alpha)-2)}-2(2h+\tau^{1/\beta(\alpha)})^{-(q(\alpha)-2)}+\sum_{k\geq3}\frac{(-1)^{k+1}}{(kh+\tau^{1/\beta(\alpha)})^{q(\alpha)-2}}\Bigr].
	\end{equation}
	In particular,
	\begin{equation}\label{eq:ea0h}
	e_{\alpha,0}(h)=-\frac{1}{h}+{h^{-(q(\alpha)-1)}}C_{1,\alpha}C_{2,\alpha}\Bigl[2-2^{3-q(\alpha)}+\sum_{k\geq3}\frac{(-1)^{k+1}}{k^{(q(\alpha)-2)}}\Bigr].
	\end{equation}

\end{lemma}

\begin{proof}
	Since the contribution of the  perimeter to the energy is clear, we just need to compute the non-local interaction.  
	Denote by
	\begin{equation*} 
	\begin{split}
	A_{\tau,\alpha}:= \int_{\R}\frac{1}{|z+\tau^{1/\beta(\alpha)}|^{q(\alpha)}}\lt[\per(E_h, [0,2h))|z|-\int_0^{2h} |\chi_{E_h}(x)-\chi_{E_h}(x+z)|\dx\rt]\dz.
	\end{split}
	\end{equation*} 
	Notice that   
   \begin{align*}
A_{\tau,\alpha}=&\int_{\R}\frac{1}{|z+\tau^{1/\beta(\alpha)}|^{q(\alpha)}} \lt(|z|-\int_0^h \chi_{E_h^c}(x+z)\dx\rt)\dz + \int_{\R}\frac{1}{|z+\tau^{1/\beta(\alpha)}|^{q(\alpha)}} \lt(|z|-\int_h^{2h} \chi_{E_h}(x+z)\dx\rt)\dz\\
=& 2 \int_{\R}\frac{1}{|z+\tau^{1/\beta(\alpha)}|^{q(\alpha)}} \lt(|z|-\int_0^h \chi_{E_h^c}(x+z)\dx\rt)\dz =4 \int_{\R^+} \frac{1}{|z+\tau^{1/\beta(\alpha)}|^{q(\alpha)}}\lt(z-\int_0^h \chi_{E_h^c}(x+z)\dx\rt)\dz,
\end{align*}
	where we have first made the change of variables $x=y+h$ and used that $x+z\in E_h$ is equivalent to $y+z\in E_h^c$ and then,  for $z<0$, we have  let  $z'=-s$ and $x'=h-x$ (so that if $x+z\in E_h^c$, also $x'+z'\in E_h^c$).

	Since  in $\R^+$ one has that  $\chi_{E_h^c}=\chi_{[0,h]^c}-\sum_{k\ge 1} \chi_{[(2k)h,(2k+1)h]}$, 
	\begin{align*}
	\int_{\R^+} \frac{1}{|z+\tau^{1/\beta(\alpha)}|^{q(\alpha)}}\lt(z-\int_{0}^{h} \chi_{E_h^c}(x+z)\dx\rt)\dz&=\int_{\R^+} \frac{1}{|z+\tau^{1/\beta(\alpha)}|^{q(\alpha)}} \lt(z-\int_{0}^{h} \chi_{[0,h]^c}(x+z)\dx\rt)\dz\\
	&+\sum_{k\ge 1}\int_{\R^+} \frac{1}{|z+\tau^{1/\beta(\alpha)}|^{q(\alpha)}}\int_0^h \chi_{[(2k)h,(2k+1)h]}(x+z)\dx\dz.
	\end{align*}
	The first term on the right-hand side can be computed as
	\[
	\int_{\R^+} \frac{1}{|z+\tau^{1/\beta(\alpha)}|^{q(\alpha)}} \lt(z-\int_{0}^{h} \chi_{[0,h]^c}(x+z)\dx\rt)\dz=\int_h^{+\infty} \frac{z-h}{|z+\tau^{1/\beta(\alpha)}|^{q(\alpha)}}=\frac{(h+\tau^{1/\beta(\alpha)})^{-(q(\alpha)-2)}}{(q(\alpha)-2)(q(\alpha)-1)},\]
	while for the second term we can use that  for $k\ge 1$,
	\begin{align*}
	\int_{\R^+}& \frac{1}{|z+\tau^{1/\beta(\alpha)}|^{q(\alpha)}}\int_0^h \chi_{[(2k)h,(2k+1)h]}(x+z)\dx\dz=\sum_{k\geq1} \int_{(2k)h}^{(2k+1)h}\int_{x-h}^x \frac{1}{|z+\tau^{1/\beta(\alpha)}|^{q(\alpha)}}\dz\dx\\
	&=C_{2,\alpha}\sum_{k\geq1}\Bigl[\frac{1}{((2k+1)h+\tau^{1/\beta(\alpha)})^{q(\alpha)-2}}+\frac{1}{((2k-1)h+\tau^{1/\beta(\alpha)})^{q(\alpha)-2}}-\frac{2}{(2kh+\tau^{1/\beta(\alpha)})^{q(\alpha)-2}}\Bigr]
	\end{align*}
	Putting together the two estimates above, we get \eqref{eq:einftyalpha} and \eqref{eq:ea0h}.
	
\end{proof}

As a consequence of the formulas \eqref{eq:einftyalpha} and \eqref{eq:ea0h} for the energies of the stripes of width $h$, one has the following upper bounds for the minimal energies achieved by stripes of the same width and distance.

\begin{corollary}\label{cor:cupper}
	For every $\alpha_0<1$ there exists a constant $\sigma_{\alpha_0}<0$ such that for all $0\leq\alpha\leq\alpha_0$ and for all $\tau>0$ it holds
	\begin{equation}
	C^*_{\alpha,\tau}=\inf_{h>0}e_{\alpha,\tau}(h)\leq\sigma_{\alpha_0}.
	\end{equation}
\end{corollary}	

\begin{proof}
	
	By direct computation, one has that $e_{\alpha,0}$ reaches its minimum at
	\begin{equation}
	h^*_{\alpha,0}=\Bigl((q(\alpha)-1)C_{1,\alpha}C_{2,\alpha}C_{3,\alpha}\Bigr)^{1/(q(\alpha)-2)},
	\end{equation}
	where for simplicity we set
	\begin{equation}
	C_{3,\alpha}=\Bigl[2-2^{3-q(\alpha)}+\sum_{k\geq3}\frac{(-1)^{k+1}}{k^{(q(\alpha)-2)}}\Bigr].
	\end{equation}
	In particular,
	\begin{equation}
	C^*_{\alpha,0}=e_{\alpha,0}(h^*_{\alpha,0})=\frac{(2-q(\alpha))C_{1,\alpha}C_{2,\alpha}C_{3,\alpha}}{((q(\alpha)-1)C_{1,\alpha}C_{2,\alpha}C_{3,\alpha})^{(q(\alpha)-1)/(q(\alpha)-2)}},
	\end{equation}
	thus in particular since $q(\alpha)\geq3-\alpha_0>2$ there exists a constant $\sigma_{\alpha_0}<0$ such that
	\begin{equation}\label{eq:cstarbound0}
	C^*_{\alpha,0}\leq \sigma_{\alpha_0}\qquad\forall\,0\leq\alpha\leq\alpha_0<1.
	\end{equation}

From formula \eqref{eq:einftyalpha} we get	
	\begin{align}
\frac{d}{d\tau}e_{\alpha,\tau}(h)&=\frac{C_{1,\alpha}C_{2,\alpha}(2-q(\alpha))\tau^{(1/\beta(\alpha))-1}}{h}\Bigl[2(h+\tau^{1/\beta(\alpha)})^{1-q(\alpha)}-2(2h+\tau^{1/\beta(\alpha)})^{1-q(\alpha)}\notag\\
&+\sum_{k\geq3}\frac{(-1)^{k+1}}{(kh+\tau^{1/\beta(\alpha)})^{q(\alpha)-1}}\Bigr]\notag\\
&\leq0,
\end{align}
where in the last inequality we used the following facts:
\begin{itemize}
	\item for $0\leq\alpha\leq\alpha_0<1$ one has that $q(\alpha)\geq3-\alpha_0>2$ and $0<\beta(\alpha)\leq1$;
	\item\[
	2(h+\tau^{1/\beta(\alpha)})^{1-q(\alpha)}-2(2h+\tau^{1/\beta(\alpha)})^{1-q(\alpha)}\geq (2-2^{\alpha-1})(h+\tau^{1/\beta(\alpha)})^{1-q(\alpha)}>0;
	\]
	\item\[
	\sum_{k\geq3}\frac{(-1)^{k+1}}{(kh+\tau^{1/\beta(\alpha)})^{q(\alpha)-1}}>0.
	\]
\end{itemize}

Hence the upper bound \eqref{eq:cstarbound0} extends to $C^*_{\alpha,\tau}$.

\end{proof}

\begin{remark}
	Notice that if $0\leq\alpha\leq\alpha_0<1$, then $e_{\alpha,\tau}(h)+\frac1h>0$ and $e_{\alpha,\tau}(h)+\frac1h\to+\infty$ as $h\to 0$. 
\end{remark}

\section{Rigidity, stability and preliminary lemmas}

In this section we collect a series of Lemmas and Propositions which contain the main estimates used in the proof of Theorem \ref{thm:main}.

Thanks to the explicit computations of constants in dependence of $\alpha$ performed in Section \ref{sec:alpha}, we can now extend most of the main lemmas of \cite{DR}[Section 7] to values $\alpha<1$. In particular, the stability estimates of Proposition \ref{lemma:stimaContributoVariazionePiccola} hold for all $\alpha<1$. The main issue is represented by the Rigidity Proposition \ref{lemma:local_rigidity_alpha}, in which by the  optimality of the rigidity as $\tau\to0$ for values $\alpha\leq0$ (see \cite{DR}[Proposition 3.2]) we will have to pass to the limit also for $\alpha\to0$ and $\alpha > 0$. This will be the only point in the proof for which the validity of Theorem \ref{thm:main} cannot be extended with this techniques for arbitrary values of $\alpha<1$.

We start with the following lemma, analogue of Remark 7.1 in \cite{DR}. 

\begin{lemma}
   \label{rmk:stimax1}
   For any $\alpha_0<1$ there exist $\eta_0 > 0$ and  $\tau_{0} > 0$ such that for every $0\leq\alpha < \alpha_0$ and  $0<\tau< \tau_0$, whenever  $E\subset \R^d$ and $s^-<s<s^+\in \partial E_{t_{i}^{\perp}}$ are three consecutive points satisfying  $\min(|s - s^- |,|s^+ -s |) <\eta_0$, then $r_{i,\alpha,\tau}(E,t_{i}^{\perp},s) > 0$. 
   
   In particular, the following estimate holds 
      \begin{equation}
   \label{eq:stimamax1_eq}
   \begin{split}
   r_{i,\alpha, \tau}(E,t^{\perp}_{i},s) \geq -1 + C_{1,\alpha}C_{2,\alpha} \min(|s-s^+ |^{-\beta(\alpha)},\tau^{-1}) + C_{1,\alpha}C_{2,\alpha}\min(|s-s^-|^{-\beta(\alpha)} , \tau^{-1})
   \end{split}
   \end{equation}
   where $C_{1,\alpha}$ and $C_{2,\alpha}$ are defined respectively in \eqref{eq:c1} and \eqref{eq:c2}. 
   
   \end{lemma}

\begin{proof}
	
	As in \cite{DR} Remark 7.1 one has that

   \begin{equation*} 
      \begin{split}
        \forall\,\rho\in(0,+\infty),\quad\text{it holds:}\quad \int_{s^-}^{s}  |\chi_{E_{t^\perp_i}}(u + \rho) -\chi_{E_{t^\perp_i}}(u) | \du \leq \min(\rho,|s - s^-|)\\
         \forall\,\rho\in(-\infty,0),\quad\text{it holds:}\quad \int_{s}^{s^+}  |\chi_{E_{t^\perp_i}}(u + \rho) -\chi_{E_{t^\perp_i}}(u) | \du \leq \min(-\rho,|s - s^+|),
      \end{split}
   \end{equation*} 
 thus 
 
 \begin{equation}\label{eq:riest1}
 r_{i,\alpha,\tau}(E,t_i^\perp,s)\geq-1+\int_{|s-s^-|}^{+\infty}(\rho-|s-s^-|)\widehat K_{\alpha,\tau}(\rho)\d\rho+\int_{-\infty}^{-|s-s^+|}(-\rho-|s-s^+|)\widehat K_{\alpha,\tau}(\rho)\d\rho.
 \end{equation}

By applying to \eqref{eq:riest1} the formula \eqref{eq:hatkformula} for $\widehat K_{\alpha,\tau}$ and then the estimate \eqref{eq:onedimestkernel} with $c=|s-s^-|$ (resp. $c=|s-s^+|$), one obtains \eqref{eq:stimamax1_eq}.

From the lower bound \eqref{eq:stimamax1_eq} the statement of the lemma follows immediately observing that for any $0\leq\alpha\leq\alpha_0<1$ one has that (point 4 of Lemma \ref{lemma:ineq2}) 
\begin{equation}
C_{1,\alpha}C_{2,\alpha}\geq C_{1,0}C_{2,0}>0
\end{equation}
and $\ba\geq\beta(\alpha_0)>0$.

\end{proof}


It is convenient to introduce  the one-dimensional analogue of \eqref{eq:ritau}. Given $E\subset \R$  a set of locally finite perimeter and let $s^-, s,s^+\in \partial E$, we define
\begin{equation}\label{eq:rtau1D}
   \begin{split}
      r_{\alpha,\tau}(E,s) := -1 & + \int_\R |\rho| \widehat{K}_{\alpha,\tau}(\rho)\d\rho  -  \int_{s^-}^{s} \int_0^{+\infty}  |\chi_{E}(\rho+ u) - \chi_{E}(u)| \widehat{K}_{\alpha,\tau} (\rho)\d\rho  \du \\ & - \int_{s}^{s^+} \int_{-\infty}^0  |\chi_{E}(\rho+ u) - \chi_{E}(u)| \widehat{K}_{\alpha,\tau} (\rho)\d\rho  \du. 
   \end{split}
\end{equation}

The quantities defined in \eqref{eq:ritau} and \eqref{eq:rtau1D} are related via $r_{i,\alpha,\tau}(E,t^\perp_i,s) = r_{\alpha,\tau}(E_{t^\perp_{i}},s)$.

In the next lemma, as in Lemma 7.5 in \cite{DR}, we determine a lower bound for the first term of the decomposition \eqref{eq:decomposition} as $(\alpha,\tau)\to (0,0)$.
\begin{lemma}
   \label{lemma:technicalBeforeLocalRigidity}
   Let $E_{0,0}, \{E_{\alpha, \tau}\}\subset \R$  be a family of sets of locally finite perimeter and $I\subset \R$ be an open bounded interval.   Moreover, assume that $E_{\alpha, \tau}\to E_{0,0}$ in $L^1(I)$.  
   If we denote by $\{k^{0,0}_{1},\ldots,k^{0,0}_{m_{0,0}}\} = \partial E_{0,0}\cap I $, then
   \begin{equation}
      \label{eq:gstr5}
      \liminf_{(\alpha,\tau)\downarrow (0,0)}\sum_{\substack{s\in \partial E_{\alpha,\tau}\\ s\in I}}r_{\alpha,\tau}(E_{\alpha,\tau},s) \geq \sum_{i=1}^{m_{0,0}-1}(-1 + {C_{1,0}C_{2,0}}|k^{0,0}_{i} - k^{0,0}_{i+1} |^{-1}),
   \end{equation}
   where $r_{\alpha,\tau}$ is defined in \eqref{eq:rtau1D}.
   \end{lemma}

\begin{proof}

   Let us denote by $\{k^{\alpha,\tau}_{1},\ldots,k^{\alpha,\tau}_{m_{\alpha,\tau}}\} = \partial E_{\alpha,\tau}\cap I$. 
   We will also denote by
   \begin{equation*}
      k^{\alpha,\tau}_0 = \sup\{ s\in \partial E_{\alpha,\tau}: s < k^{\alpha,\tau}_1\} \qquad\text{and}\qquad 
      k^{\alpha,\tau}_{m_{\alpha,\tau}+1} = \inf\{ s\in \partial E_{\alpha,\tau}: s > k^{\alpha,\tau}_{m_{\alpha,\tau}}\} .
   \end{equation*}
   Denote by $A$ the \rhs of \eqref{eq:gstr5}. 
   From \eqref{eq:stimamax1_eq}, one has that 
   \[
   r_{\alpha,\tau}(E_{\alpha,\tau},k^{\alpha,\tau}_{i}) \geq -1  + C_{1,\alpha}C_{2,\alpha} \bigl[\min(|k^{\alpha,\tau}_{i} - k^{\alpha,\tau}_{i+1} |^{-\beta(\alpha)},\tau^{-1})+\min(|k^{\alpha,\tau}_{i} - k^{\alpha,\tau}_{i-1} |^{-\beta(\alpha)},\tau^{-1})\bigr].
   \] 
   Then fixing $\alpha_0<1$ there exists ${\eta}=\eta(A,\alpha_0)$ such that for every $\alpha \leq {\alpha}_0$ and for every 
    $|s - t | < {\eta} $:
    \begin{align}\label{eq:ineqeta}
    C_{1,\alpha}C_{2,\alpha}|s-t |^{-\beta(\alpha)} 
    \geq {C_{1,0}C_{2,0}} |s-t |^{-\beta({\alpha}_0)} \geq \bar C\eta^{-\beta(\alpha_0)}\geq A.
    \end{align}
    In \eqref{eq:ineqeta} we have used the following: point 4 of Lemma \ref{lemma:ineq2} and the fact that $\beta(\alpha)$ is decreasing in $\alpha$ and $\beta(\alpha)>0$ for all $\alpha<1$.
 
   Thus, there exist $\eta$ and $\bar{\tau}> 0$ such that for every $\tau \leq\bar{\tau}$, $\alpha \leq {\alpha_0}$, whenever 
   \begin{equation*}
      \begin{split}
         \min_{i\in \{0,\ldots,{m_{\alpha,\tau}}\}}|k^{\alpha,\tau}_{i+1}- k^{\alpha,\tau}_{i}| < \eta
      \end{split}
   \end{equation*}
   then
   \begin{equation*}
       \sum_{\substack{s\in \partial E_{\alpha,\tau}\\ s\in I}}r_{\alpha,\tau}(E_{\alpha,\tau},s) \geq A. 
   \end{equation*}
   Hence, assume there exists a subsequence $\alpha_k,\tau_{k}$ such that $|k^{\alpha_k,\tau_k}_{i+1}- k^{\alpha_k,\tau_{k}}_{i}| > \eta$ for all $i\leq m_{\alpha_k,\tau_k}$. 
   Up to  relabeling, let us assume that it holds true  for the whole sequence  $\{E_{\alpha,\tau}\}$.

   Since $\min_{i} | k^{\alpha,\tau}_{i+1} - k^{\alpha,\tau}_{i} | > \eta$ the convergence $E_{\alpha,\tau}\to E_{0,0}$ in  $L^1(I)$ can be upgraded to the convergence of the boundaries, namely one has that there exists a $\bar{\tau}$, $\bar{\alpha}$ such that for $\tau\leq\bar{\tau}$, $\alpha \leq \bar{\alpha}$ it holds $m_{\alpha,\tau}=\#(\partial E_{\alpha,\tau}\cap I) = \#(\partial E_{0,0}\cap I)=m_{0,0}$ and $k^{\alpha,\tau}_{i} \to k^{0,0}_{i}$. 
 Since  $C_{1,\alpha} C_{2, \alpha}$ and $\beta(\alpha)$ are continuous in $\alpha$ and  because of the convergence of the boundaries, we have that (recalling that $\beta(0)=1$)
   \begin{equation}
      \label{eq:gstr13}
      \begin{split}
         \liminf_{(\alpha,\tau)\downarrow (0,0)}\sum_{\substack{s\in \partial E_{\alpha,\tau}\\ s\in I}}r_{\alpha,\tau}(E_{\alpha,\tau},s) &\geq \liminf_{(\alpha,\tau)\downarrow (0,0)}\sum_{i=0}^{m_{\alpha,\tau}} \big( - 1 + C_{1,\alpha}C_{2,\alpha}
         \min(|k^{\alpha,\tau}_{i} - k^{\alpha,\tau}_{i+1}|^{-\beta(\alpha)},\tau^{-1})  \big)
         \\ & \geq\sum_{i=0}^{m_{0,0}} \big( - 1 + {C_{1,0}C_{2,0}}|k^{0,0}_{i} - k^{0,0}_{i+1}|^{-1}\big).
      \end{split}
   \end{equation}
\end{proof}

The next proposition is at the base of symmetry breaking at scale $l$: one a square of size $l$, if $\tau$ and $\alpha$ are sufficiently close to $0$, a bound on the energy corresponds to a bound on the $L^1$-distance to the unions of stripes.

\begin{proposition}[Local Rigidity] 
   \label{lemma:local_rigidity_alpha}
    For every $M > 1,l,\delta > 0$, there exist $0<{\alpha}_1<1$, $\tau_1>0$ and $\bar{\eta} >0$ 
     such that whenever $0\leq\alpha <{\alpha}_1$, $0<\tau< {\tau}_1$  and $\bar F_{\alpha, \tau}(E,Q_{l}(z)) < M$ for some $z\in [0,L)^d$ and $E\subset\R^d$ $[0,L)^d$-periodic, with $L>l$, then it holds $D_{\eta}(E,Q_{l}(z))\leq\delta$ for every $\eta < \bar{\eta}$. Moreover $\bar{\eta}$ can be chosen independently  of $\delta$.  Notice that ${\alpha}_1$, ${\tau}_1$ and $\bar{\eta}$ are independent of $L$.
\end{proposition} 

\begin{proof}
	The proof aims at reaching a contradiction in case the assumptions are not satisfied, as in \cite{DR}. In this case however, we will have to perform a double limit in $\tau$ and $\alpha$.
	
\textbf{Step 1}

Assume by contradiction that there exist $M>1, l >0, \delta > 0$ and sequences $(\alpha_k, \tau_k), (\eta_k), (L_k)$, $(z_k), (E_{\alpha_k, \tau_k})$  
such that 
   \begin{enumerate}[(i)]
   \item    $(\alpha_k,\tau_k) \downarrow (0,0) $, $L_k > l$, $\eta_k \downarrow 0$, $z_k\in[0,L_k)^d$; 
    \item the family of sets $E_{\alpha_k,\tau_k}$ is $[0,L_k)^d$-periodic
    \item one has that $D_{\eta_k}(E_{\alpha_k,\tau_k},Q_{l}(z_k)) > \delta$ and $\bar{F}_{\alpha_k,\tau_k}(E_{\alpha_k,\tau_k},Q_{l}(z_k)) < M$. 
   \end{enumerate}

Given that $\eta \mapsto D_{\eta}(E,Q_l(z)) $ is monotone increasing, we can fix $\bar{\eta}$  sufficiently small instead of $\eta_{k}$ with $D_{\bar\eta}(E_{\alpha_k,\tau_k},Q_l(z_k)) >  \delta$. 
In particular, $\bar{\eta}$ will be chosen at the end of the proof depending only on $M,l$. 

W.l.o.g. (taking e.g. $E_{\alpha_k,\tau_k}-z_k$ instead of $E_{\alpha_k,\tau_k}$) we can assume there exists $z\in\R^d$ such that $z_k=z$ for all $k\in\N$. 

\textbf{Step 2}
   Thanks to Lemma ~\ref{rmk:stimax1}, one has that $\sup_{k}\per_1(E_{\alpha_k, \tau_k},Q_{l}(z)) < +\infty$. 
    Thus up to subsequences there exists a set of finite perimeter $E_{0,0}$ such that 
   $$ \lim_{(\alpha_k,\tau_k) \rightarrow (0,0)}E_{\alpha_k,\tau_{k}} = E_{0,0}$$ 
	   in $L^1(Q_l(z))$ with $D_{\bar{\eta}}(E_{0,0},Q_{l}(z))> \delta$. 
 For  every $\alpha_k$ there exists a set of finite perimeter $E_{\alpha_k, 0}$ and a subsequence $\tau_{n_k}$ such that 
   $$\lim_{\tau_{n_k} \rightarrow 0} E_{\alpha_k,\tau_{n_k}} = E_{\alpha_k, 0}$$ in $L^1(Q_l(z))$. 
   We have that $\per(E_{\alpha_k,0}) \leq \liminf_{\tau_{n_k} \rightarrow 0} \per(E_{\alpha_k, \tau_k}) \leq C$, so there exists also $\bar E_{0,0}$ and a subsequence $\alpha_{m_k}$ such that
   $$ \lim_{\alpha_{m_k} \rightarrow 0} E_{\alpha_{m_k},0} = \bar{E}_{0,0}$$
   in $L^1(Q_l(z))$. By lower semicontinuity, $\bar E_{0,0}$ has finite perimeter.   
   Since 
$$\lim_{(\alpha_{m_k},\tau_{n_{m_k}})\rightarrow (0,0)} E_{\alpha_{m_k},\tau_{n_{m_k}}} = E_{0,0}, $$ 
   We get that $E_{0,0} = \bar{E}_{0,0}$.

   We relabel the subsequence $(\alpha_{m_k}, \tau_{n_{m_k}}) $ as $(\alpha, \tau) $. Then we can assume that 
   $$E_{\alpha,\tau} 
   \underset{\tau \rightarrow 0} {\rightarrow} E_{\alpha, 0} \qquad \textrm{and} 
   \qquad 
   E_{\alpha,0} 
   \underset{\alpha \rightarrow 0}{\rightarrow} E_{0,0} $$ 
    in $L^1(Q_l(z))$. 

\textbf{Step 3}

Let $J_i$ be the interval $(z_i - l/2, z_i + l/2)$. By Lebesgue's theorem, there exists a subsequence of $(\alpha, \tau)$ such that for almost every $t_i^{\perp} \in Q_l^{\perp} (z_i^{\perp})$  one has that $E_{\alpha, \tau, t_i^{\perp}} \cap J_i$ converges to $E_{0,0, t_i^{\perp}} \cap J_i$ in $L^1(Q_l(z))$. By $E_{\alpha, \tau, t_i^{\perp}}$ we denote the one-dimensional slice of $E_{\alpha,\tau}$ in direction $e_i$ with respect to the point $t_i^\perp\in Q^\perp_{l}(z_i^\perp)$.  Using the definition of $\bar{F}_{\alpha,\tau}$ given in \eqref{eq:fbartau} and the fact that $v_{i,\alpha,\tau} \geq 0$ we get following bounds for the functional $\bar{F}_{\alpha,\tau}$
       \begin{align}
         \label{eq:gstr7}
          M &\geq  \bar{F}_{\alpha,\tau}(E_{\alpha,\tau},Q_l(z)) \notag \\ 
          & \geq
         \frac{1}{l^d }\sum_{i =1}^{d}\int_{Q_{l}^{\perp}(z_{i}^{\perp})} \sum_{\substack{s\in \partial E_{\alpha,\tau, t^\perp_i}\\ s\in J_i}} 
         r_{i,\alpha,\tau}(E_{\alpha,\tau},t_{i}^{\perp},s) \dt^\perp_{i} + \int_{Q_{l}(z)} w_{i,\alpha,\tau}(E_{\alpha,\tau},t_{i}^{\perp},t_{i})\dt_{i}^{\perp}\dt_{i}.
      \end{align}
      
\textbf{Step 4}

Thanks to the fact that $w_{i,\alpha,\tau}\geq0$ and by Fatou Lemma we can estimate the first term of \eqref{eq:gstr7}  as follows 
        \begin{align}
          l^d M &\geq  \liminf_{(\alpha,\tau)\downarrow (0,0)}\sum_{i =1}^{d}\int_{Q_{l}^{\perp}(z_{i}^{\perp})} \sum_{\substack{s\in \partial E_{\alpha,\tau,t^\perp_i}\\ s\in J_i}} r_{i,\alpha,\tau}(E_{\alpha, \tau},t_{i}^\perp,s) \dt_{i}^{\perp}  \notag \\         
          & \geq 
            \sum_{i =1}^{d}\int_{Q_{l}^{\perp}(z_{i}^{\perp})} \liminf_{(\alpha,\tau)\downarrow (0,0)}\sum_{\substack{s\in \partial E_{\alpha,\tau,t^\perp_i}\\ s\in J_i}} 
            r_{i,\alpha,\tau}(E_{\alpha,\tau},t_{i}^{\perp},s) \dt^{\perp}_{i}.
            \end{align}
            Now we observe that we can apply Lemma \ref{lemma:technicalBeforeLocalRigidity} and obtain that 
            
            \begin{align}\label{eq:r0est}
            l^dM&\geq 
            \sum_{i =1}^{d}\int_{Q_{l}^{\perp}(z_{i}^{\perp})} \sum_{\substack{s\in \partial E_{0,0, t^\perp_i}\\ s\in J_i}} 
            \big(-1  + C_{1,0}C_{2,0}(s^+-s)^{-1} + C_{1,0}C_{2,0}(s-s^-)^{-1}\big) \dt_{i}^\perp,
      \end{align}

      \textbf{Step 5} Let us now proceed to estimate the second term in \eqref{eq:gstr7}.
       Given $t\in \R^d$, we will denote by $t_{i} = \scalare{t, e_i} e_{i} $, $t_{i}^\perp =  t - t_{i}$ and we recall that
          \begin{equation*}
       {w}_{i,\alpha, \tau}(E,t_{i}^{\perp},t_{i}) = \frac{1}{d}\int_{\R^d}  
       f_{E}(t_{i}^{\perp},t_{i},\zeta_{i}^{\perp},\zeta_{i}) \Ka(\zeta)  \d\zeta
       \end{equation*}
       with
   \begin{equation*}
      \begin{split}
         f_{E}(t^{\perp}_i,t_i,t'^\perp_{i},t'_i):=|\chi_{E}(t_{i}^\perp +t_i+ t'_{i}) - \chi_{E}(t_i + t^{\perp}_{i} ) | | \chi_{E}(t_{i}^\perp +t_i+ t'^{\perp}_{i}) - \chi_{E}(t_i + t^{\perp}_{i} )  |.
      \end{split}
   \end{equation*}

Let $0<\tau<\tau'$.

Then, the monotonicity of ${K}_{\alpha, \tau}$ w.r.t. $\tau$ stated in point 3 of Lemma \ref{lemma:propertiesKhat} and the positivity of $f_{E}$, lead us to the following inequality
\begin{align}\label{eq:wmon}
{w}_{i,\alpha, \tau'}(E_{\alpha, \tau},t_{i}^{\perp},t_{i})
\leq
{w}_{i,\alpha, \tau}(E_{\alpha, \tau},t_{i}^{\perp},t_{i}). 
\end{align} 
By Fatou Lemma and the lower semicontinuity of $f_E$ w.r.t. $L^1$ convergence of sets we obtain
\begin{align}
\liminf_{\tau \downarrow 0}
         \int_{Q_{l}(z)} w_{i,\alpha,\tau'}(E_{\alpha,\tau},t_{i}^{\perp},t_{i}) \dt_{i}^{\perp}\dt_{i}
         & \geq 
      \int_{\R^d}K_{\alpha,\tau'}(\zeta)
\liminf_{\tau \downarrow 0}      
      \int_{Q_l(z)} \frac{1}{d}  
      f_{E_{\alpha, \tau}}(t_{i}^{\perp},t_{i},\zeta_{i}^{\perp},\zeta_{i})   \d\zeta \notag \\
          & \geq 
         \int_{Q_{l}(z)} w_{i,\alpha,\tau'}(E_{\alpha,0},t_{i}^{\perp},t_{i}) \dt_{i}^{\perp}\dt_{i}.    \label{eq:wmon2}    
\end{align}
By \eqref{eq:wmon} and \eqref{eq:wmon2} we can deduce the following
\begin{align}
         \liminf_{\tau \downarrow 0}
         \int_{Q_{l}(z)} w_{i,\alpha,\tau}(E_{\alpha,\tau},t_{i}^{\perp},t_{i}) \dt_{i}^{\perp}\dt_{i}
         &\geq
         \liminf_{\tau'\downarrow0}\liminf_{\tau\downarrow0}
          \int_{Q_{l}(z)} w_{i,\alpha,\tau'}(E_{\alpha,\tau},t_{i}^{\perp},t_{i}) \dt_{i}^{\perp}\dt_{i}\notag\\
         &\geq
         \lim_{\tau' \downarrow 0}
         \int_{Q_{l}(z)} w_{i,\alpha,\tau'}(E_{\alpha,0},t_{i}^{\perp},t_{i}) \dt_{i}^{\perp}\dt_{i}.
\end{align}

     This allow us to deduce that
   \begin{align}
         \liminf_{(\alpha,\tau) \downarrow (0,0)} \int_{Q_{l}(z)} w_{i,\alpha,\tau}(E_{\alpha,\tau},t_{i}^{\perp},t_{i}) \dt_{i}^{\perp}\dt_{i}  
         &\geq
         \liminf_{\alpha \downarrow 0}
         \liminf_{\tau \downarrow 0}
         \int_{Q_{l}(z)} w_{i,\alpha,\tau}(E_{\alpha,\tau},t_{i}^{\perp},t_{i}) \dt_{i}^{\perp}\dt_{i}
         \\
         &\geq 
         \liminf_{\alpha \downarrow 0}
         \lim_{\tau \downarrow 0}
         \int_{Q_{l}(z)} w_{i,\alpha,\tau}(E_{\alpha,0},t_{i}^{\perp},t_{i}) \dt_{i}^{\perp}\dt_{i}
         \label{eq:monotonictyintau} \\
         &\geq 
         \liminf_{\alpha \downarrow 0}
         \int_{Q_{l}(z)} w_{i,\alpha,0}(E_{\alpha,0},t_{i}^{\perp},t_{i}) \dt_{i}^{\perp}\dt_{i}
         \label{eq:monotoneconvergencetheoremintau}
\end{align}
where in the last inequality we use again the monotonicity of ${K}_{\alpha, \tau}$ stated in point 3 of Lemma \ref{lemma:propertiesKhat} and the monotone convergence theorem. 


Let us define 
\begin{align}
w^{<1}_{i,\alpha,0}(E,t_{i}^{\perp},t_{i}) &:= 
   \frac{1}{d}\int_{|\zeta| < 1}  
      f_{E}(t_{i}^{\perp},t_{i},\zeta_{i}^{\perp},\zeta_{i}) K_{\alpha,0}(\zeta)  \d\zeta, 
\\
   w^{>1}_{i,\alpha,0}(E,t_{i}^{\perp},t_{i}) 
   &:= 
   w_{i,\alpha,0}(E,t_{i}^{\perp},t_{i})
   -
   w^{<1}_{i,\alpha,0}(E,t_{i}^{\perp},t_{i}).
\end{align}
to handle the singularity of ${K}_{\alpha,0}$ in the origin.
 By the monotonicity property 4 in Lemma \ref{lemma:propertiesKhat} one has that whenever  $\alpha\leq\alpha'$ 
\begin{align}
\label{eq:alphaMonotonicityK}
{w}^{<1}_{i,\alpha',0}(E_{\alpha, 0},t_{i}^{\perp},t_{i})
\leq
{w}^{<1}_{i,\alpha, 0}(E_{\alpha,0},t_{i}^{\perp},t_{i}). 
\end{align} 
Hence by Fatou Lemma and the lower semiconinuity of $f_E$ w.r.t. $L^1$ convergence of sets
\begin{align}
\liminf_{\alpha \downarrow 0}
         \int_{Q_{l}(z)} w^{<1}_{i,\alpha',0}(E_{\alpha,0},t_{i}^{\perp},t_{i}) \dt_{i}^{\perp}\dt_{i}
         & \geq 
      \int_{\R^d} K_{\alpha',0}(\zeta)
\liminf_{\alpha \downarrow 0}      
      \int_{Q_l(z)} \frac{1}{d}  
      f_{E_{\alpha, 0}}(t_{i}^{\perp},t_{i},\zeta_{i}^{\perp},\zeta_{i})   \dt_i^\perp\dt_i\d\zeta \notag \\
          & \geq 
         \int_{Q_{l}(z)} w^{<1}_{i,\alpha',0}(E_{0,0},t_{i}^{\perp},t_{i}) \dt_{i}^{\perp}\dt_{i}   
         \label{eq:alphaliminf}      
\end{align}
where we used the fact that  $E_{\alpha, 0}\to E_{0, 0}$ in $L^1(Q_l(z))$.
Together \eqref{eq:alphaMonotonicityK} and \eqref{eq:alphaliminf} imply
\begin{align*}
         \liminf_{\alpha'\downarrow 0}
         \int_{Q_{l}(z)} w^{<1}_{i,\alpha',0}(E_{\alpha',0},t_{i}^{\perp},t_{i}) \dt_{i}^{\perp}\dt_{i}
         \geq
         \lim_{\alpha' \downarrow 0}
         \int_{Q_{l}(z)} w^{<1}_{i,\alpha',0}(E_{0,0},t_{i}^{\perp},t_{i}) \dt_{i}^{\perp}\dt_{i}.
\end{align*}
We use property 4 in Lemma \ref{lemma:propertiesKhat} and Lebesgue monotone convergence Theorem to get 
\begin{align}\label{eq:walphaliminf}
\lim_{\alpha' \downarrow 0}
         \int_{Q_{l}(z)} w^{<1}_{i,\alpha',0}(E_{0,0},t_{i}^{\perp},t_{i}) \dt_{i}^{\perp}\dt_{i} \geq 
\int_{Q_{l}(z)} w^{<1}_{i,0,0}(E_{0,0},t_{i}^{\perp},t_{i}) \dt_{i}^{\perp}\dt_{i}
\end{align}
As for the part of the cross interaction term containing $w^{>1}_{i,\alpha',0}$ we have that 
\begin{align}
\liminf_{\alpha' \downarrow 0}
         \int_{Q_{l}(z)} w^{>1}_{i,\alpha',0}(E_{\alpha',0},t_{i}^{\perp},t_{i}) \dt_{i}^{\perp}\dt_{i}
         &\geq\liminf_{\alpha' \downarrow 0}
         \int_{Q_{l}(z)} w^{>1}_{i,0,0}(E_{\alpha',0},t_{i}^{\perp},t_{i}) \dt_{i}^{\perp}\dt_{i} 
         \label{eq:walphazero_K} \\
         &\geq
         \int_{Q_{l}(z)} w^{>1}_{i,0,0}(E_{0,0},t_{i}^{\perp},t_{i}) \dt_{i}^{\perp}\dt_{i}
         \label{eq:walphazero_cont}
\end{align}
thanks to the continuity of  $K_{\alpha,0}$ in $\alpha$ and the dominated convergence theorem in \eqref{eq:walphazero_K}
and l.s.c. of $f_E$ w.r.t. $L^1$ convergence of sets in  \eqref{eq:walphazero_cont}.

Collecting \eqref{eq:alphaliminf}, \eqref{eq:walphaliminf} 
and \eqref{eq:walphazero_cont} we have that 
\begin{align}
\liminf_{\alpha \downarrow 0}
         & \int_{Q_{l}(z)} w_{i,\alpha,0}(E_{\alpha,0},t_{i}^{\perp},t_{i}) \dt_{i}^{\perp}\dt_{i} = \notag\\
         & =
          \liminf_{\alpha \downarrow 0}
         \Bigl[\int_{Q_{l}(z)} w^{>1}_{i,\alpha,0}(E_{\alpha,0},t_{i}^{\perp},t_{i}) \dt_{i}^{\perp}\dt_{i}
         +
         \int_{Q_{l}(z)} w^{<1}_{i,\alpha,0}(E_{\alpha,0},t_{i}^{\perp},t_{i}) \dt_{i}^{\perp}\dt_{i}\Bigr]\notag
         \\
         &\geq  \liminf_{\alpha \downarrow 0}
         \int_{Q_{l}(z)} w^{>1}_{i,\alpha,0}(E_{\alpha,0},t_{i}^{\perp},t_{i}) \dt_{i}^{\perp}\dt_{i}
         +
          \liminf_{\alpha \downarrow 0}
         \int_{Q_{l}(z)} w^{<1}_{i,\alpha,0}(E_{\alpha,0},t_{i}^{\perp},t_{i}) \dt_{i}^{\perp}\dt_{i}\notag
         \\
         &\geq  
         \int_{Q_{l}(z)} w^{>1}_{i,0,0}(E_{0,0},t_{i}^{\perp},t_{i}) \dt_{i}^{\perp}\dt_{i}
         +
         \int_{Q_{l}(z)} w^{<1}_{i,0,0}(E_{0,0},t_{i}^{\perp},t_{i}) \dt_{i}^{\perp}\dt_{i}\notag
         \\
     &\geq \frac{1}{d}\int_{Q_{l}(z)}\int_{Q_{l}(z)} 
         f_{E_{0}} (t^\perp_i, t_i, t'^\perp_i - t'^\perp_{i},t_i - t'_i) K_{0}(t-t')\dt \dt'.\label{eq:w0est}
   \end{align}
   where at the end we reduce to the cube.


    \textbf{Step 6}

      Similarly to \cite{GR} and \cite{DR} one can show that \eqref{eq:stimamax1_eq} implies that
      \begin{equation*}
         \begin{split}
            \per_1(E_{\alpha,\tau} ,Q_l(z)) \leq l^d C_{1,\alpha}C_{2,\alpha}  \max(1,\bar{F}_{\alpha,\tau}(E_{\alpha,\tau},Q_l(z))) \leq l^d C_{1,\alpha}C_{2,\alpha} \max(1,M) \leq l^d  C_{1,\alpha}C_{2,\alpha} M,
         \end{split}
      \end{equation*}
      thus from the lower semicontinuity of the perimeter and the continuity in $\alpha$ of the constants $C_{1,\alpha}$ and $C_{2,\alpha}$ we have that
      \begin{equation}\label{eq:perest}
         \begin{split}
            \per_1(E_{0,0},Q_l(z))) \leq \liminf_{(\alpha,\tau)\downarrow (0,0)}\per_1(E_{\alpha,\tau},Q_l(z))) \leq C_{1,0}C_{2,0} l^d M.
         \end{split}
      \end{equation}
      In particular, from \eqref{eq:r0est}, \eqref{eq:perest} and \eqref{eq:w0est} one obtains that
      \begin{align}
         \label{eq:gstr10}
         &   \sum_{i =1}^{d}\int_{Q_{l}^{\perp}(z_{i}^{\perp})} \sum_{\substack{s\in \partial E_{0,0,t^\perp_i}\\ s\in J_i}} ((s^+-s)^{-1} +(s-s^-)^{-1} )          \dt^\perp_{i} \leq  l^d M 
            \\  &\frac1d\sum_{i=1}^{d}\int_{Q_{l}(z)} \int_{Q_{l}(z)}f_{E_{0,0}} (t^\perp_i, t_i, t'^\perp_i - t'^\perp_{i},t_i - t'_i) K_{0,0}(t-t')\dt \dt' \leq  l^d M \label{eq:gstr38}.
      \end{align}

   Now that we have the above bounds for  $\alpha=0$  (i.e., $p(\alpha)=d+2$) and $\tau=0$ the same reasoning as in \cite{DR}[Proposition 3.2] can be applied in order to obtain that \eqref{eq:gstr10} and \eqref{eq:gstr38}  hold  if and only if $E_{0,0}\cap Q_{l}(z)$ is a union of stripes. It is indeed sufficient to substitute $[0,L)^d$ with $Q_l(z)$ in the proof.
   
   Moreover, since the \lhs  of \eqref{eq:gstr10} diverges  for stripes with minimal width  tending to zero, one has that there exists $\bar{\eta} = \bar{\eta}(M,l)$ such that $D_{\bar{\eta}}(E_{0,0},Q_l(z))  = 0$.  This contradicts the assumption that $D_{\bar{\eta}}(E_{0,0}, Q_l(z)) >\delta$,  which was assumed at the beginning of the proof.
    
\end{proof}

In particular, one has the following 
\begin{corollary}\label{cor:gammaconv}
Let $\alpha\leq\alpha_0<1$ and $0<\tau\leq\tau_0\ll1$. One has that the following holds:
\begin{itemize}
\item Let $\{E_{\alpha,\tau}\}$ be a sequence such that $\sup_{\alpha,\tau} \bar{F}_{\alpha, \tau}(E_{\alpha,\tau}, Q_l(z)) < \infty$. 
Then the sets $E_{\alpha,\tau}$ converge in $L^1$  up to  subsequences to some set $E_{0,0}$ of finite perimeter and 
\begin{align}
\liminf_{(\alpha,\tau) \rightarrow (0,0)} \bar{F}_{\alpha, \tau}(E_{\alpha,\tau}, Q_l(z)) \geq \bar{F}_{0, 0}(E_{0,0}, Q_l(z)) . 
\label{eq:liminfLocalGamma}
\end{align}
\item For every set $E_{0,0}$ with $\bar{F}_{0, 0}(E_{0,0}, Q_l(z)) < + \infty$, there exists a sequence $\{E_{\alpha,\tau}\}$ converging in $L^1$ to $E_{0,0}$ and such that
\begin{align}
\limsup_{(\alpha,\tau) \rightarrow (0,0)} \bar{F}_{\alpha, \tau}(E_{\alpha,\tau}, Q_l(z)) = \bar{F}_{0,0}(E_{0,0}, Q_l(z)) . 
\end{align} 
\end{itemize}
\end{corollary}

The following proposition

\begin{proposition}[Local Stability]
   \label{lemma:stimaContributoVariazionePiccola}
         Let  $(t^{\perp}_{i}+se_i)\in (\partial E) \cap [0,l)^d$, $\alpha_0<1$ and  $\eta_{0}$, $\tau_0$ depending on $\alpha_0$ as in Lemma~\ref{rmk:stimax1}. Then there exist ${\tau_2},\varepsilon_2$ (independent of $l$) such that for every $0<\tau < {\tau_2}$, $0\leq\alpha <\alpha_0$ and $0<\varepsilon < {\varepsilon_2}$ the following holds: assume that 
         \begin{enumerate}[(a)]
            \item $\min(|s-l|, |s|)> \eta_0$ (i.e. the boundary point $s$ in the slice of $E$ is sufficiently far from the boundary of the cube)
            \item $D^{j}_{\eta}(E,[0,l)^d)\leq\frac {\varepsilon^d} {16 l^d}$ for some $\eta> 0$ and  with $j\neq i$ (i.e. $E\cap [0,l)^d$ is close to stripes with boundaries orthogonal to $e_j$ for some $j\neq i$)
         \end{enumerate}
         Then 
         \[r_{i,\alpha,\tau}(E,t^{\perp}_{i},s) + v_{i,\alpha,\tau}(E,t^{\perp}_{i},s) \geq 0.\] 
\end{proposition}
\begin{proof}

   By Remark~\ref{rmk:stimax1},  $\eta_{0}$ and  $\tau_0>0$  are such that if $0<\tau<\tau_0$, $0\leq \alpha < \alpha_0$ and $s^-<s<s^+$ are three consecutive points in  $\partial E_{t_i^\perp}$ (see \eqref{eq:s+s-}) then  
   \begin{equation*} 
      \begin{split}
         \min(|s-  s^- |, |s^+ -s |) <\eta_{0} \quad\Rightarrow\quad r_{i,\alpha,\tau}(E,t_{i}^{\perp},s) > 0. 
      \end{split}
   \end{equation*} 
   
      Thus, since $v_{i,\alpha,\tau}\geq0$, we may assume that $\min(|s - s^- |, |s^+ -s |) \geq \eta_{0}$.

   We choose now $0<{\varepsilon_2}<\eta_0$ and $0<{\tau_2}<\tau_0$ as in Point 2 of Lemma \ref{lemma:ineq2}.  Hence 
   \begin{equation}\label{eq:epstaualpha}
   \begin{split}
   \frac{7C/16 \varepsilon^{d+1}}{ \varepsilon^{p(\alpha)} + \tau^{p(\alpha)/\beta(\alpha)} } \geq 1 , \qquad \text{where $C$ is such that } \qquad  
   K_{\alpha,\tau}(\zeta)\geq \frac{C}{|\zeta |^{p(\alpha)} + \tau^{p(\alpha)/\beta(\alpha)}},
   \end{split}
   \end{equation}
   for every $\eps<\eps_2$, $\tau <{\tau}_2$ and $0\leq\alpha\leq\alpha_0<1$. 
   
 By condition $(a)$ there exists a cube $Q^\perp_{{\varepsilon}_1}(t'^\perp_i)\subset \R^{d-1}$ of size ${\varepsilon_1}$,  such that $t^\perp_i\in Q^{\perp}_{{\varepsilon}_1}(t'^\perp_{i})$ and  $(s-{\varepsilon}_1,s+{\varepsilon}_1)\times Q^{\perp}_{{\varepsilon}_1}(t'^\perp_i) \subset [0,l)^d$. 
   
   By the fact that $r_{i,\alpha,\tau}\geq-1$ and by definition of $v_{i,\tau,\alpha}$  one has that
   \begin{equation}\label{eq:rv} 
   \begin{split}
   r_{i,\alpha,\tau}(E,t^{\perp}_{i},s)+ v_{i,\alpha,\tau}(E,t^{\perp}_{i},s)\geq -1 + \int_{s -{\varepsilon_1 }}^{s + {\varepsilon_1}}\int_{-2{\varepsilon_1}}^{2\varepsilon_1} \int_{Q^\perp_{{\varepsilon_1}}(t'^\perp_{i})} f_{E}(t^{\perp}_{i},t_{i},\zeta^{\perp}_{i},\zeta_{i}) K_{\alpha,\tau}(\zeta) \d\zeta_i^{\perp}\d\zeta_i\dt_{i} . 
   \end{split}
   \end{equation} 
   
   One reduces then to estimate the r.h.s. of  \eqref{eq:rv}. In order to do so, one proceeds as in the proof of Lemma 6.1 in \cite{DR}, using now that by  condition (b) $E$ is $L^1$-close to stripes with boundaries orthogonal to $e_j$ on the cube $[0,l)^d$ instead of $[0,L)^d$. In this way one obtains

   \begin{equation}\label{eq:7.35} 
   \begin{split}
   \int_{s-{\varepsilon_1}}^{s}\int_{-2 \varepsilon_1} ^{2{ \varepsilon_1 }}  \int_{Q^{\perp}_{{\varepsilon_1}}(t'^{\perp}_{i})} f_{E}(t^{\perp}_{i},u,\zeta^{\perp}_{i},\zeta_{i}) K_{\alpha,\tau}(\zeta) \d\zeta \du &\geq\frac{7C\slash 16 {\varepsilon_2}^{d+1}}{{\varepsilon_2}^{p(\alpha)} +\tau^{ p(\alpha)/\beta(\alpha)}}.
   \end{split}
   \end{equation} 
   
   Then, by \eqref{eq:rv}, \eqref{eq:7.35} and \eqref{eq:epstaualpha}, one concludes that
   \[ r_{i,\alpha,\tau}(E,t^{\perp}_{i},s)+ v_{i,\alpha,\tau}(E,t^{\perp}_{i},s)\geq0.\]

\end{proof}

The following lemma contains  a one-dimensional estimate needed in the proof  of Lemma~\ref{lemma:stimaLinea}, which  is the analogue of Lemma 7.7 in \cite{DR}.

In the proof, as  in \cite{DR}, one uses a periodic extension argument (which gives the error term $C_0(\eta_0,\alpha_0)$) and then the fact that for periodic one-dimensional sets the energy contribution in \eqref{eq:gstr40} is bigger than or equal to the contribution of periodic stripes of width $h^*_{\alpha,\tau}$, namely $C^*_{\alpha,\tau}$ times the length of the interval $I$. 
\begin{lemma}
	\label{lemma:1D-optimization}
	Let $\alpha_0<1$. Then, there exists $C_0=C_0(\eta_0,\alpha_0)$ with $\eta_0$ as in Lemma \ref{rmk:stimax1} such that the following holds.
	Let $E\subset \R$  be a set of locally finite perimeter and $I\subset \R$ be an open interval. 
	Let $s^-, s$ and $s^+$ be three consecutive points on the boundary of $E$ and $r_{\alpha,\tau}(E,s)$ defined as in \eqref{eq:rtau1D}.
	Then for all $0\leq\alpha\leq\alpha_0$,  $0<\tau< \tau_0$, where $\tau_0$ is given in Remark~\ref{rmk:stimax1}, it holds
	\begin{equation}
	\label{eq:gstr40}
	\sum_{\substack{s \in \partial E\\ s \in I}} r_{\alpha,\tau}(E,s) \geq C^*_{\alpha,\tau} |I| - C_0.
	\end{equation}
\end{lemma}

\begin{proof}
	The proof is analogous to the one of Lemma 7.7 in \cite{DR}. We report here only the steps where an explicit dependence on $\alpha$ of the various constants has to be analysed. 
	
	Let us denote by $k_1< \ldots< k_m $ the  points of $\partial E \cap I$, and 
	
	\begin{equation*}
	k_0 = \sup\{ s\in \partial E: s < k_1\} \qquad\text{and}\qquad 
	k_{m+1} = \inf\{ s\in \partial E: s > k_m\} 
	\end{equation*}
	
	W.l.o.g.,  we may assume that  $r_{\alpha,\tau}(E,k_1) < 0$ and that $r_{\alpha,\tau}(E,k_m) < 0$.

	Because of Lemma~\ref{rmk:stimax1}, the fact that $r_{\alpha,\tau}(E,k_1) < 0$ and $r_{\alpha,\tau}(E,k_m) < 0$ implies that there exists $\eta_0=\eta_0(\alpha_0)>0$ (for all  $0<\tau \leq \tau_0$ and $0\leq\alpha\leq\alpha_0$) such that 
	\begin{equation*}
	\min(|k_1 - k_0 |, |k_2 - k_1 |, |k_{m-1} - k_m |,|k_{m+1} - k_m |) > \eta_0. 
	\end{equation*}

	We now prove the following claim:
	\begin{equation}
	\label{eq:gstr4}
	\sum_{i =1}^{ m } r_{\alpha,\tau}(E,k_i) \geq \sum_{i =1}^{ m } r_{\alpha,\tau}(E',k_i)  - \bar C_0 
	\end{equation}
	where $E'$ is obtained by extending periodically $E$ with the pattern contained in $E \cap (k_1,k_m)$ and $\bar C_0  = \bar C_0(\eta_0,\alpha_0) > 0$. 
	The construction of $E'$ can be done as follows: if $m$ is odd we repeat periodically $E\cap (k_1,k_m)$, and if $m$ is even we repeat periodically $(k_1-\eta_{0}, k_m)$. 
	
	Thus we have constructed a set $E'$ which is periodic of period $k_m-k_1$ or $k_m- k_1 + \eta_0$. Therefore
	
	\begin{equation}
	\label{eq:ggstr1}
	\sum_{i=1}^m r_{\alpha,\tau}(E',k_i) \geq C^{*}_{\alpha,\tau} (k_m - k_1 +\eta_0) \geq  C^*_{\alpha,\tau} |I| - \tilde C _0,
	\end{equation}
	where $ \tilde C_0 = \tilde C_0(\eta_0,\alpha_0)$.

	As in Lemma 7.7 of \cite{DR} one has that 
		\begin{equation*}
	\Big|\sum_{i=1}^m r_{\alpha,\tau}(E,k_i) - \sum_{i=1}^m r_{\alpha,\tau}(E',k_i)\Big| \leq \int_{k_0}^{k_m}\int_{k_{m} + \eta_0}^{\infty} \widehat{K}_{\alpha,\tau}(u-v) \du\dv
	+ \int_{k_1}^{k_{m+1}}\int_{-\infty}^{k_{1} - \eta_0} \widehat{K}_{\alpha,\tau}(u-v) \dv\du.
	\end{equation*}

	Thus by using the integrability of $\widehat K_{\alpha,\tau}$ for $0\leq\alpha\leq\alpha_0<1$, we have that 
	\begin{equation*}
	\Big|\sum_{i=1}^m r_{\alpha,\tau}(E,k_i) - \sum_{i=1}^m r_{\alpha,\tau}(E',k_i)\Big| \leq C_0(\eta_0,\alpha_0).
	\end{equation*}

	Since for every periodic set we have that $C^{*}_{\alpha,\tau}$ is the infimum of all the energy densities for periodic sets of any period (due to point 2. of lemma \ref{lemma:propertiesKhat} and the reflection positivity technique), we have the desired result. 
\end{proof}

The following lemma extends to all $\alpha<1$ Lemma 7.9 in \cite{DR}.

\begin{lemma}
	\label{lemma:stimaLinea}
	Let $\alpha_0<1$ and ${\eps_2},{\tau_2}>0$ as in Lemma \ref{lemma:stimaContributoVariazionePiccola}. Let $\delta=\eps^d/(16l^d)$ with $0<\eps\leq{\eps_2}$, $0<\tau\leq{\tau_2}$ and $l>C_0/(-C^*_{\alpha,\tau})$ for $0\leq\alpha\leq\alpha_0$, where $C_0$ is the constant appearing in Lemma \ref{lemma:1D-optimization}. Let $t_i^\perp\in[0,L)^{d-1}$ and $\eta>0$.

	The following hold: there exists a constant $C_1$ independent of $l$ (but depending on the dimension and on $\eta_0$ as in Lemma \ref{rmk:stimax1}) such that  
	\begin{enumerate}[(i)]
		\item Let $J\subset \R$ an interval  such that for every $s\in J$ one has that  $D^{j}_{\eta}(E,Q_{l}(t^{\perp}_{i}+se_i))\leq \delta$ with $j\neq i$. 
		Then
		\begin{equation}
		\label{eq:gstr20}
		\begin{split}
		\int_{J} \bar{F}_{i,\alpha,\tau}(E,Q_{l}(t^{\perp}_{i}+se_i))\ds \geq - \frac{C_1}{l}.
		\end{split}
		\end{equation}
		Moreover, if $J = [0,L)$, then 
		\begin{equation}
		\label{eq:gstr21}
		\begin{split}
		\int_{J} \bar{F}_{i,\alpha,\tau}(E,Q_{l}(t^{\perp}_{i}+se_i))\ds \geq0.
		\end{split}
		\end{equation}
		\item Let $J = (a,b)\subset \R$. 
		If for $s=a$ and $s=b$ it holds $D_\eta^j(E,Q_{l}(t^{\perp}_i+se_i)) \leq \delta$ with $j\neq i$, then 
		\begin{equation}
		\label{eq:gstr27}
		\begin{split}
		\int_{J} \bar{F}_{i,\alpha,\tau}(E,Q_{l}(t^{\perp}_{i}+se_i))\ds \geq | J| C^{*}_{\alpha,\tau} -\frac{C_1} l,
		\end{split}
		\end{equation}
		otherwise
		\begin{equation}
		\label{eq:gstr36}
		\begin{split}
		\int_{J} \bar{F}_{i,\alpha,\tau}(E,Q_{l}(t^{\perp}_{i}+se_i))\ds \geq | J| C^{*}_{\alpha,\tau} - C_1l.
		\end{split}
		\end{equation}
		Moreover, if $J = [0,L)$, then
		\begin{equation}
		\label{eq:gstr28}
		\begin{split}
		\int_{J} \bar{F}_{i,\alpha,\tau}(E,Q_{l}(t^{\perp}_{i}+se_i))\ds \geq | J| C^{*}_{\alpha,\tau}.
		\end{split}
		\end{equation}
	\end{enumerate}
\end{lemma}

	The proof of this lemma is similar to the one of Lemma 7.9 in \cite{DR} so we omit it. The only point where the dependence on $\alpha$ plays a role is in the constant $\eta_0=\eta_0(\alpha_0)$ introduced in Lemma \ref{rmk:stimax1} and in the optimal energy densities $C^*_{\alpha,\tau}$.

	The next lemma is the analogue of Lemma 7.11 in \cite{DR} and gives a lower bound on the energy in the case almost all the volume of $Q_l(z)$ is filled by $E$ or $E^c$ (this will be the case on the set $A_{-1}$ defined in \eqref{a1}).
	\begin{lemma}
		\label{lemma:stimaQuasiPieno}
		Let $\alpha_0<1$ and let $E$ be a set of locally finite perimeter  such that $\min(|Q_{l}(z)\setminus E|, |E\cap Q_{l}(z) |)\leq {\delta} l^d$, for some $\delta>0$. Then 
		\begin{equation*}
		\begin{split}
		\bar F_{\alpha,\tau} (E,Q_{l}(z)) \geq -\frac {\delta d } {\eta_0 },
		\end{split}
		\end{equation*}
		where $\eta_0=\eta_0(\alpha_0)$ is defined in Lemma \ref{rmk:stimax1}.
	\end{lemma}

\section{Proof of Theorem~\ref{thm:main}}\  
The strategy of the proof of Theorem \ref{thm:main} follows closely the one of Theorem 1.4 in \cite{DR}. The main difference w.r.t. \cite{DR} is that we additionally have to take care of the dependence of the parameters appearing in the various estimates w.r.t. $\alpha$. In this way we can extend the result of \cite{DR} to a range of positive values of $\alpha$ (i.e., $p<d+2$, $d+2-p\ll1$).

\subsection{Setting the parameters}\label{Ss:param}
The sets defined in the proof and the main estimates will depend on a set of parameters $l,\delta,\rho,M,\alpha, \eta$ and $\tau$.  Our aim now is to fix such parameters, making explicit their dependence on each other. We will refer to such choices during the proof of the main theorem.

\begin{enumerate}
	\item We first fix $\alpha_0<1$ and then find $\eta_0,\tau_0$ as in Lemma \ref{rmk:stimax1} and $ \sigma_{\alpha_0}$ as in Corollary \ref{cor:cupper}.
	
	\item  We  fix then $l>0$ s.t.
	\begin{equation}\label{eq:lfix}
	l>\max \Big \{ \frac{dC(d,\eta_0)}{-\sigma_{\alpha_0}}, \frac{C_0}{-\sigma_{\alpha_0}}\Big\}, 
	\end{equation}
	where $C(d,\eta_0)$ is the constant (depending only on the dimension $d$ and on $\alpha_0<1$) defined in \eqref{eq:ca0} and appearing in \eqref{eq:toBeShown_integral}, and
	$C_0=C_0(\eta_0,\alpha_0)$ is the constant which appears in the statement of Lemma \ref{lemma:1D-optimization}. 
	
	In particular, thanks to Corollary \ref{cor:cupper} one has that
	\begin{equation}\label{eq:lfix2}
		l>\max \Big \{ \frac{dC_d}{-C^*_{\alpha,\tau}}, \frac{C_0}{-C^*_{\alpha,\tau}}\Big\}, 
	\end{equation}
	for all $0\leq\alpha\leq\alpha_0$ and for all $\tau>0$.
	
	\item We  find  the parameters  ${\varepsilon}_2 = {\varepsilon}_2(\eta_0,\tau_0)$ and ${\tau}_2 = {\tau}_2(\eta_0, \tau_0)$ as in Proposition \ref{lemma:stimaContributoVariazionePiccola}.
	
	\item We consider then  $\varepsilon \leq {\varepsilon}_2$, $\tau \leq {\tau}_2$ as in Lemma~\ref{lemma:stimaLinea}. We define   $\delta$ as $\delta  = \frac{\varepsilon^d}{16}$. Moreover,  by choosing $\varepsilon$ sufficiently small we can additionally assume that
	\begin{equation}\label{eq:deltafix2}
	D^i_{\eta}(E,Q_l(z))\leq\delta\text{ and }D^j_\eta(E,Q_l(z))\leq\delta,\:i\neq j\quad\Rightarrow\quad\min\{|E\cap Q_l(z)|,|E^c\cap Q_l(z)|\}\leq l^{d-1}.  
	\end{equation}
	The above  follows from Remark~\ref{rmk:lip} (ii). 
	
	\item By Remark \ref{rmk:lip} (i), we then fix
	\begin{equation}\label{eq:rhofix}
	\rho\sim\delta l. 
	\end{equation}
	in such a way that  for any $\eta$ the following holds
	\begin{equation}\label{eq:rhofix2}
	\forall\,z,z'\text{ s.t. }D_{\eta}(E,Q_l(z))\geq\delta,\:|z-z'|_\infty\leq\rho\quad\Rightarrow\quad D_\eta(E,Q_l(z'))\geq\delta/2.
	\end{equation}

	\item Then we fix $M$ such that
	\begin{equation}
	\label{eq:Mfix}
	\frac{M\rho}{2d}>C_1l,
	\end{equation}
	where $C_1=C_1(\eta_0)$ is the constant appearing in Lemma \ref{lemma:stimaLinea}.
	
	\item By applying Proposition~\ref{lemma:local_rigidity_alpha}, we obtain $0<{\alpha}_1<1$ with ${\alpha}_1={\alpha}_1(M,l,\delta)$, $\bar\eta=\bar{\eta}(M,l)$  and ${\tau}_1 = {\tau}_1(M,l,\delta/2)$.  Thus we fix
	\begin{equation}\label{eq:etafix}
	0<\eta<\bar{\eta},\quad \bar\eta=\bar{\eta}(M,l)
	\end{equation}
	and
	\begin{equation}
	0<\bar{\alpha}\leq\min\{\alpha_0,\alpha_1\}.
	\end{equation}
	\item Finally, we choose  $\bar\tau>0$ s.t.
	\begin{equation}
	\label{eq:taufix0}
	\begin{split}
	\bar\tau<\tau_0,  \qquad \tau_0 \text{ as in Lemma~\ref{rmk:stimax1},}
	\end{split}
	\end{equation}
	\begin{equation}
	\label{eq:taufix1}
	\bar\tau<{\tau}_2, \,{\tau}_2\text{ as in Proposition \ref{lemma:stimaContributoVariazionePiccola} and Lemma \ref{lemma:stimaLinea}},
	\end{equation}
	\begin{equation}
	\label{eq:taufix2}
	\bar\tau<{\tau}_1, \text{ ${\tau}_1$ as in Proposition \ref{lemma:local_rigidity_alpha} depending on $M,l,\delta/2$}.
	\end{equation}
	
\end{enumerate}

Let $E$ be a minimizer of $\FtL$. By $[0,L)^d$-periodicity  of $E$ we will denote by $[0,L)^d$  the cube of size $L$ with the usual identification of the boundary.

\subsection{ Decomposition of $[0,L)^d$} 

This and the following section, once we have set the parameters as in Section \ref{Ss:param}, proceed as in \cite{DR}. We report them here by completeness.

Let us now consider any $L>l$ of the form $L=2kh^*_{\alpha,\tau}$, with $k\in\N$, $0\leq\alpha\leq\bar\alpha$ and $\tau\leq\bar{\tau}$ as in Section \ref{Ss:param}. In this section we recall a decomposition of $[0,L)^d$ introduced in \cite{DR} Section 7. We will have that
$[0,L)^d =A_{-1}\cup A_0 \cup \ldots \cup A_d$ where
\begin{itemize}
	\item $A_i$ with $i > 0$ are the set of points $z$ such that there is only one direction $e_i$ such that $E_\tau\cap Q_{l}(z)$ is close to stripes with boundaries orthogonal to~$e_i$. 
	\item $A_{-1}$ is the set of points $z$ such that $E_\tau\cap Q_{l}(z)$ is close both to stripes with boundaries orthogonal to $e_i$ and to stripes with boundaries orthogonal to $e_j$ for some $i\neq j$.  In particular, by Remark \ref{rmk:lip} (ii) one has that either $|E_\tau\cap Q_l(z)|\ll l^d$ or $|E_{\tau}^c\cap Q_l(z)|\ll l^d$.
	\item $A_{0}$ is the set of points $z$ where none of the above points is true.
\end{itemize}

The aim is then to show that $A_0\cup A_{-1} = \emptyset$ and  that there exists only one $A_i$ with $i >  0$.

Let us give the precise definitions of the sets $A_i$, $i\in\{-1,0,1,\ldots,d\}$.

We preliminarily define
\begin{equation*}
\begin{split}
\tilde{A}_{0}:= \insieme{ z\in [0,L)^d:\ D_{\eta}(E,Q_{l}(z)) \geq \delta }.
\end{split}
\end{equation*}
Hence, by the choice of $\delta,M$ made in Section \ref{Ss:param} and by Proposition \ref{lemma:local_rigidity_alpha}, for every $z\in \tilde{A}_{0}$ one has that $\bar{F}_{\alpha,\tau}(E,Q_{l}(z)) > M$. 

Let us denote by $\tilde{A}_{-1}$ the set
\begin{equation*}
\begin{split}
\tilde{A}_{-1}: = \insieme{z\in [0,L)^d: \exists\, i,j \text{ with } i\neq j \text{ \st }\, D^{i}_{\eta} (E,Q_{l}(z))\leq\delta , D^{j}_{\eta} (E,Q_{l}(z)) \leq \delta }.
\end{split}
\end{equation*}

Since $\delta$ satisfies \eqref{eq:deltafix2}, when $z\in \tilde{A}_{-1}$, then one has that $\min (|E\cap Q_{l}(z)|, |Q_{l}(z)\setminus E|) \leq  l^{d-1} $.
Thus, using  Lemma~\ref{lemma:stimaQuasiPieno} with $\delta=1/l$, one has that
\begin{equation*}
\begin{split}
\bar{F}_{\alpha,\tau}(E, Q_{l}(z)) \geq  -\frac{d}{l\eta_0}.
\end{split}
\end{equation*}

Now we show that the sets $\tilde A_0$ and $\tilde A_{-1}$ can be enlarged while keeping analogous properties.

By the choice of $\rho$ made in \eqref{eq:rhofix}, \eqref{eq:rhofix2} holds, namely for every $z\in \tilde{A}_{0}$ and $|z- z' |_\infty\leq\rho$ one has that $D_{\eta}(E,Q_{l}(z')) > \delta/2$.  

Moreover, let now $z'$ be such that $|z- z' |_\infty\leq 1$ with $z\in \tilde{A}_{-1}$. It is not difficult to see that if $|Q_{l}(z)\setminus E | \leq l^{d-1}$ then $|Q_{l}(z')\setminus E| \lesssim l^{d-1}$. Thus from Lemma~\ref{lemma:stimaQuasiPieno}, one has that
\begin{equation}
\label{eq:tildeC}
\begin{split}
\bar{F}_{\alpha,\tau}(E, Q_{l}(z')) \geq -\frac{\tilde C_d}{l\eta_0}.
\end{split}
\end{equation}

The above observations motivate the following definitions
\begin{align}
A_{0} &:= \insieme{ z' \in [0,L)^d: \exists\, z \in \tilde{A}_{0}\text{ with }|z-z'|_{\infty}  \leq \rho }\label{a0}\\
A_{-1} &:= \insieme{ z' \in [0,L)^d: \exists\, z \in \tilde{A}_{-1}\text{ with }|z-z' |_{\infty}  \leq 1 },\label{a1}
\end{align}

By the choice of the parameters and the observations above, for every $z\in A_{0}$ one has that $\bar{F}_{\alpha,\tau}(E,Q_{l}(z)) > M$ and for every $z\in A_{-1}$, $\bar{F}_{\alpha,\tau}(E,Q_{l}(z)) \geq-\tilde C_d/(l\eta_0)$.

Let us denote by $A:= A_{0}\cup A_{-1}$. 

The set $[0,L)^d\setminus A$ has the following property: for every $z\in [0,L)^d\setminus A$, there exists $i\in \{ 1,\ldots,d\}$ such that $D^{i}_{\eta}(E,Q_{l}(z)) \leq \delta$ and for every $k\neq i$ one has that $D^{k}_{\eta}(E,Q_{l}(z)) > \delta$.

Given  that $A$ is closed, we consider the connected components $\mathcal C_{1},\ldots,\mathcal C_{n}$ of $[0,L)^d\setminus A$.  The sets $\mathcal C_{i}$ are path-wise connected. 
Moreover, given a connected component $\mathcal C_{j}$ one has that there exists  $i$ such that $D^{i}_{\eta}(E,Q_{l}(z)) \leq \delta$ for every $z\in\mathcal  C_{j}$  and for every $k\neq i$ one has that $D^{k}_{\eta}(E,Q_{l}(z)) > \delta$.  
We will say that $\mathcal C_j$ is oriented in direction $e_i$ if there is a point in $z\in \mathcal C_j$ such that $D^{i}_\eta(E,Q_{l}(z)) \leq \delta$. 
Because of the above being oriented along direction $e_{i}$ is well-defined.

We will denote by $A_{i}$ the union of the connected  components $\mathcal C_{j}$ such that $\mathcal C_{j}$ is oriented along the direction $e_{i}$. 

The following properties are those which are mostly used in the following.

\begin{enumerate}[(a)]
	\item The sets $A=A_{-1}\cup A_{0}$, $A_{1}$, $A_{2}$, $\ldots, A_d$  form a partition of $[0,L)^d$. 
	\item The sets $A_{-1}, A_{0}$ are closed and $A_{i}$, $i>0$, are open.  
	\item For every $z\in A_{i}$, we have that $D^{i}_{\eta}(E,Q_{l}(z)) \leq \delta$. 
	\item  There exists $\rho$ (independent of $L,\tau$) such that  if $z\in A_{0}$, then $\exists\,z'$ s.t. $Q_{\rho}(z')\subset A_{0}$ and $z \in Q_{\rho}(z')$. If $z\in A_{-1}$ then $\exists\,z'$ s.t. $Q_{1}(z')\subset A_{-1}$ and $z \in Q_{1}(z')$. 
	\item For every $z\in {A}_{i}$ and $z'\in {A}_{j}$ one has that there exists a point $\tilde{z}$ in the segment connecting $z$ to $z'$ lying in ${A}_{0}\cup A_{-1}$. 
\end{enumerate}

\subsection{Main Estimate}

Let $B = \bigcup_{i> 0}A_{i}$.
Our aim is to show that for every $i>0$, $0\leq\alpha\leq\bar{\alpha}$ and $0<\tau\leq \bar{\tau}$ as in Section \ref{Ss:param}, the following holds
\begin{equation}
\label{eq:toBeShown_integral}
\begin{split}
\frac{1}{L^d}\int_{B} \bar{F}_{i,\alpha,\tau}(E,Q_{l}(z))\dz + \frac1{d L^d} \int_{A}\bar{F}_{\alpha,\tau}(E,Q_{l}(z)) \dz  \geq \frac{C^{*}_{\alpha,\tau}|A_{i}|}{L^d} - C(d,\eta_0) \frac{|A|}{l L^d}
\end{split}
\end{equation}
where $C(d,\eta_0)$ depends on the dimension $d$ and on $\eta_0=\eta_0(\alpha_0)$.

Indeed, once \eqref{eq:toBeShown_integral} is given we can sum it over $i>0$ and obtain that, by \eqref{eq:gstr14} and since $l$ satisfies \eqref{eq:lfix2}, 
\begin{equation}\label{eq:ineqfinal}
\begin{split}
\Fcal_{\alpha,\tau,L}(E) & \geq \sum_{i=1}^{d}\frac{1}{L^d} \int_{[0,L)^d} \bar{F}_{i,\alpha,\tau}(E,Q_{l}(z))  \dz \geq  \frac{C^{*}_{\alpha,\tau}}{L^d} \sum_{i=1}^{d} |A_i|  - \frac{dC_d|A|}{lL^d}
\\ & \geq C^{*}_{\alpha,\tau} - C^*_{\alpha,\tau} \frac{|A|} {L^d} - \frac{dC(d,\eta_0)}{lL^d} |A| \\
&\geq C^{*}_{\alpha,\tau}.
\end{split}
\end{equation}
In the third inequality we have used that  $|A | + \sum_{i=1}^{d} |A_{i}| =  |[0,L)^d | = L^d$, $C^{*}_{\alpha,\tau} < 0$ and in the last inequality we exploited \eqref{eq:lfix2}. 

Notice that, in the inequality \eqref{eq:ineqfinal}, equality holds only if $|A|=0$ and therefore by $(v)$ only if there is just one $A_i$, $i>0$ with $|A_i|>0$. This proves the statement of Theorem \ref{thm:main}.

Indeed, let us consider 
\begin{align}
\frac{1}{L^d}\int_{[0,L)^d}\bar F_{\alpha,\tau}(E,Q_l(z))\dz&=\frac{1}{L^d}\int_{[0,L)^d}\bar F_{i,\alpha,\tau}(E,Q_l(z))\dz\label{eq:fi}\\
&+\frac{1}{L^d}\sum_{j\neq i}\int_{[0,L)^d}\bar F_{j,\alpha,\tau}(E,Q_l(z))\dz\label{eq:fj}
\end{align}

We apply now Lemma~\ref{lemma:stimaLinea} with $j =i$ and slice the cube $[0,L)^d$ in direction $e_i$. 
From \eqref{eq:gstr21}, one has that \eqref{eq:fj} is nonnegative and strictly positive unless the set $E$ is a union of stripes with boundaries orthogonal to $e_i$. 
On the other hand, from \eqref{eq:gstr28}, one has the \rhs of \eqref{eq:fi} is minimized by a periodic union of stripes with boundaries orthogonal to $e_i$ and with width $h^*_{\alpha,\tau}$. Thus, periodic stripes of distance and width $h^*_{\alpha,\tau}$ are optimal.

Hence, the rest of the proof is devoted to proving \eqref{eq:toBeShown_integral}.

First we notice that \eqref{eq:toBeShown_integral} follows from the analogous statement on the slices, namely that for every $t^{\perp}_{i}\in [0,L)^{d-1}$, it holds

\begin{equation}
\label{eq:toBeShown_slice}
\begin{split}
\frac{1}{L^d} \int_{B_{t^{\perp}_{i}}} \bar{F}_{i,\alpha,\tau}(E,Q_{l}(t^{\perp}_{i}+se_i))\ds + \frac1{dL^d} \int_{A_{t^{\perp}_{i}}}\bar{F}_{\alpha,\tau}(E,Q_{l}(t^{\perp}_{i}+se_i)) \ds  \geq \frac{C^{*}_{\alpha,\tau}|A_{i,t^{\perp}_{i}}|}{L^d} - C(d,\eta_0) \frac{|A_{t^\perp_i}|}{l L^d}
\end{split}
\end{equation}

Indeed by integrating \eqref{eq:toBeShown_slice} over $t^{\perp}_{i}$ we obtain \eqref{eq:toBeShown_integral}. Thus we reduce to prove \eqref{eq:toBeShown_slice}.

Notice that, from conditions $(b)$ and $(e)$ in the previous section, $B_{t^{\perp}_{i}}$ is a finite union of intervals. Moreover, by $(d)$, for every point that does not belong to $B_{t^\perp_{i}}$ there is a neighbourhood of fixed positive size that is not included in $B_{t^\perp_i}$. 
Let $\{ I_{1},\ldots,I_{n}\}$ such that $\bigcup_{j=1}^{n} I_{j} = B_{t_i^\perp}$ with $I_j \cap I_{k} = \emptyset$ whenever $j\neq k$. 
We can further assume that $I_{i} \leq I_{i+1}$, namely that for every $s\in I_{i}$ and $s'\in I_{i+1}$ it holds $s \leq s'$. 
By construction there exists $J_{j} \subset A_{t^{\perp}_{i}}$ such that $I_{j}\leq  J_{j} \leq I_{j+1}$.  

Whenever $J_j \cap A_{0,t_i^\perp}\neq\emptyset$,  we have that $|J_j | > \rho$  and whenever $J_{j} \cap A_{-1,t^\perp_i}\neq \emptyset $ then $|J_{i}| > 1$.

Thus we have that
\begin{equation*}
\begin{split}
\frac{1}{L^d} \int_{B_{t^{\perp}_{i}}} \bar{F}_{i,\alpha,\tau}(E,&Q_{l}(t^{\perp}_{i}+se_i)) \ds  + 
\frac{1}{d L^d} \int_{A_{t^{\perp}_{i}}} \bar{F}_{\alpha,\tau}(E,Q_{l}(t^{\perp}_{i}+se_i)) \ds 
\\ & \geq \sum_{j=1}^n  \frac{1}{L^d}\int_{I_{j}} \bar{F}_{i,\alpha,\tau}(E,Q_{l}(t^{\perp}_{i}+se_i)) \ds
+ \frac{1}{dL^d}\sum_{j=1}^n \int_{J_{j}} \bar{F}_{\alpha,\tau}(E,Q_{l}(t^{\perp}_i+se_i)) \ds 
\\ & \geq \frac{1}{L^d}\sum_{j=1}^n \Big( \int_{I_{j}} \bar{F}_{i,\alpha,\tau}(E,Q_{l}(t^{\perp}_{i}+se_i)) \ds
+ \frac{1}{2d} \int_{J_{j-1}\cup J_j} \bar{F}_{\alpha,\tau}(E,Q_{l}(t^{\perp}_i+se_i)) \ds\Big),
\end{split}
\end{equation*}
where in the second inequality we have used periodicity and $J_0:=J_n$.

Let first $I_{j} \subset A_{i,t_i^\perp}$.  
By construction, we have that $\partial I_{j}\subset A_{t^\perp_i}$. 

If $\partial I_{j}\subset A_{-1,t^\perp_i}$, by using our choice of parameters, namely \eqref{eq:lfix2} and \eqref{eq:taufix1}, we can apply \eqref{eq:gstr27} in Lemma~\ref{lemma:stimaLinea} and obtain
\begin{equation*}
\begin{split}
\frac{1}{L^d}\int_{I_{j}} \bar{F}_{i,\alpha,\tau}(E,Q_{l}(t^{\perp}_{i}+se_i))\ds  \geq \frac{1}{L^d}\Big(| I_j| C^{*}_{\alpha,\tau} -\frac{C_1} l\Big).
\end{split}
\end{equation*}

If $\partial I_j \cap A_{0, t^\perp_i}\neq \emptyset$, by using our choice of parameters, namely \eqref{eq:lfix} and \eqref{eq:taufix1}, we can apply \eqref{eq:gstr36} in Lemma~\ref{lemma:stimaLinea}, and obtain
\begin{equation*}
\begin{split}
\frac{1}{L^d}\int_{I_{j}} \bar{F}_{i,\alpha,\tau}(E,Q_{l}(t^{\perp}_{i}+se_i))\ds \geq\frac{1}{L^d}\Big(| I_j| C^{*}_{\alpha,\tau}-C_1 l\Big).
\end{split}
\end{equation*}

On the other hand, if $\partial I_j \cap A_{0,t^\perp_i}\neq \emptyset$, we have that either $J_{j}\cap A_{0,t^\perp_i}\neq \emptyset$ or $J_{j-1}\cap A_{0,t^\perp_i}\neq\emptyset$. Thus
\begin{equation*}
\begin{split}
\frac{1}{2dL^d}\int_{J_{j-1}} \bar{F}_{\alpha,\tau}(E,Q_{l}(t^{\perp}_{i}+se_i)) \ds & + \frac{1}{2dL^d}\int_{J_{j}} \bar{F}_{\alpha,\tau}(E,Q_{l}(t^{\perp}_{i}+se_i)) \ds  \\ &\geq  \frac{M\rho}{2dL^d}  - \frac{|J_{j-1}\cap A_{-1,t^\perp_i} |\tilde C_d}{2dl\eta_0 L^d} - \frac{|J_{j}\cap A_{-1,t^\perp_i} |\tilde C_d}{2dl \eta_0L^d},
\end{split}
\end{equation*}
where $\tilde C_d$ is the  constant in \eqref{eq:tildeC}.

Since $M$ satisfies \eqref{eq:Mfix}, in both cases $\partial I_{j}\subset A_{-1,t^\perp_i}$ or $\partial I_{j}\cap A_{0,t^\perp_i}\neq \emptyset$, we have that 
\begin{equation*}
\begin{split}
\frac{1}{L^d}\int_{I_{j}} \bar{F}_{i,\alpha,\tau}(E,Q_{l}(t^\perp_i+se_i)) \ds &+ 
\frac{1}{2dL^d}\int_{J_{j-1}}  \bar{F}_{\alpha,\tau}(E,Q_{l}(t^{\perp}_{i}+se_i))\ds
+ \frac{1}{2dL^d}\int_{J_{j}}  \bar{F}_{\alpha,\tau}(E,Q_{l}(t^{\perp}_{i}+se_i))\ds\\
&\geq \frac{C^{*}_{\alpha,\tau} |I_{j} |}{L^d} - \frac{|J_{j-1}\cap A_{-1,t^\perp_i}|\tilde C_d}{2dl\eta_0L^d}
- \frac{|J_{j}\cap A_{-1,t^\perp_i}|\tilde C_d}{2dl\eta_0L^d}.
\end{split}
\end{equation*}

If $I_{j} \subset A_{k,t^\perp_i}$ with $k \neq i$ from  Lemma~\ref{lemma:stimaLinea} Point (i) it holds
\begin{equation*}
\begin{split}
\frac 1{L^d}\int_{I_{j}} \bar{F}_{i,\alpha,\tau}(E,Q_{l}(t^{\perp}_{i}+se_i)) \ds\geq  - \frac{C_1}{lL^d}.
\end{split}
\end{equation*}

In general for every $J_{j}$  we have that 
\begin{equation*}
\begin{split}
\frac{1}{dL^d}\int_{J_{j}} \bar{F}_{\alpha,\tau}(E,Q_{l}(t^{\perp}_{i}+se_i))\, \ds \geq   \frac{|J_{j}\cap A_{0,t^\perp_i} | M}{dL^d} - \frac{\tilde C_d}{dl\eta_0L^d }|J_{j}\cap A_{-1,t^\perp_i}|. 
\end{split}
\end{equation*}

For $I_{j}\subset A_{k,t^\perp_i}$ such that $(J_j \cup J_{j-1})\cap A_{0,t^\perp_i}\neq \emptyset$ with $k\neq i$, we have that 
\begin{equation*}
\begin{split}
\frac{1}{L^d}\int_{I_{j}} \bar{F}_{i,\alpha,\tau}(E,Q_{l}(t^\perp_i+se_i)) \ds &+ 
\frac{1}{2dL^d}\int_{J_{j-1}}  \bar{F}_{\alpha,\tau}(E,Q_{l}(t^{\perp}_{i}+se_i))\ds
+ \frac{1}{2dL^d}\int_{J_{j}}  \bar{F}_{\alpha,\tau}(E,Q_{l}(t^{\perp}_{i}+se_i))\ds\\
&\geq -\frac{C_1}{lL^d} + \frac{M\rho}{2dL^d} - \frac{|J_{j-1}\cap A_{-1,t^\perp_i}|\tilde C_d}{2dl\eta_0L^d}
- \frac{|J_{j}\cap A_{-1,t^\perp_i}|\tilde C_d}{2dl\eta_0L^d}.
\\ &\geq
- \frac{|J_{j-1}\cap A_{-1,t^\perp_i}|\tilde C_d}{2dl\eta_0L^d}
- \frac{|J_{j}\cap A_{-1,t^\perp_i}|\tilde C_d}{2dl\eta_0L^d}.
\end{split}
\end{equation*}
where the last inequality is true due to \eqref{eq:Mfix}.

For $I_{j}\subset A_{k,t^\perp_i}$ such that $(J_j \cup J_{j-1})\subset  A_{-1,t^\perp_i}$ with $k\neq i$, we have that 
\begin{equation*}
\begin{split}
\frac{1}{L^d}\int_{I_{j}} \bar{F}_{i,\alpha,\tau}(E,Q_{l}(t^\perp_i+se_i)) \ds &+ 
\frac{1}{2dL^d}\int_{J_{j-1}}  \bar{F}_{\alpha,\tau}(E,Q_{l}(t^{\perp}_{i}+se_i))\ds
+ \frac{1}{2dL^d}\int_{J_{j}}  \bar{F}_{\alpha,\tau}(E,Q_{l}(t^{\perp}_{i}+se_i))\ds\\
&\geq -\frac{C_1}{lL^d}   - \frac{|J_{j-1}\cap A_{-1,t^\perp_i}|\tilde C_d}{2dl\eta_0L^d}
- \frac{|J_{j}\cap A_{-1,t^\perp_i}|\tilde C_d}{2dl\eta_0L^d}.
\\ &\geq
- \max\Big(C_1,\frac{\tilde C_d}{\eta_0d}\Big)\bigg(\frac{|J_{j-1}\cap A_{-1,t^\perp_i}|}{lL^d}
+ \frac{|J_{j}\cap A_{-1,t^\perp_i}|}{lL^d}\bigg).
\end{split}
\end{equation*}
where in the last inequality we have used that $|J_j\cap A_{-1,t^\perp_i}|\geq1, \,|J_{j-1}\cap A_{-1,t^\perp_i}|\geq1$.

Summing over $j$, and taking
\begin{equation}\label{eq:ca0}
C(d,\eta_0)=\max\Big(C_1,\frac{\tilde C_d}{\eta_0d}\Big), 
\end{equation} one obtains \eqref{eq:toBeShown_slice} as desired.

\section*{Acknowledgments}

This work has been supported by the Research Training Group (Graduiertenkolleg) 2339 “Interfaces, Complex Structures, and Singular Limits in Continuum Mechanics” of the German Research Foundation (DFG). The author is grateful to Sara Daneri and Eris Runa for many valuable discussions.

\end{document}